\documentclass{article}
\usepackage[margin=1in]{geometry}
\usepackage{wrapfig}
\usepackage{subfig}
\usepackage{graphicx}
\usepackage{float}
\usepackage{bbm}
\usepackage[title]{appendix}

\usepackage{mathtools}     
\usepackage{natbib}
\usepackage{amsmath, amssymb, enumerate, amsthm}
\usepackage{amsfonts}
\usepackage{mathrsfs}
\usepackage{xcolor}
\usepackage[colorlinks=true,
            pdfpagemode=UseNone,
            urlcolor=blue,
            citecolor=blue,
            linkcolor=violet,
            bookmarks=true,
            backref=page
            ]{hyperref}

\newcommand{\ind}[1]{\mathbbm{1}\left\{#1\right\}}
\newcommand{\indtw}{\mathbbm{1}}

\newcommand{\iid}{i.i.d.}
\newcommand{\rdd}{\mathbb{R}^{d}}
\newcommand{\re}{\mathbb{R}}
\newcommand{\limn}{\lim_{n\rightarrow\infty}}
\newcommand{\sam}[2]{\mathbb{#1}_{#2}}
\newcommand{\samplespace}{(\rdd)^n}
\newcommand{\norm}[1]{\left\lVert#1\right\rVert}

\newcommand{\cond}{\stackrel{\text{d}}{\rightarrow}}
\newcommand{\D}{{\rm D}}
\newcommand{\R}{{\rm R}}

\DeclareMathOperator{\HD}{HD}
\DeclareMathOperator{\IRW}{IRW}
\DeclareMathOperator{\SMD}{SD}
\DeclareMathOperator{\PD}{PD}

\newcommand{\conp}{\stackrel{\text{p}}{\rightarrow}}

\DeclareMathOperator*{\argmax}{ {\rm argmax}}

\DeclareMathOperator*{\argmin}{ {\rm argmin}}

\DeclarePairedDelimiter\floor{\lfloor}{\rfloor}

\DeclareMathOperator{\med}{Med}

\DeclareMathOperator{\iqr}{IQR}
\DeclareMathOperator{\mad}{MAD}
\newcommand{\GS}{\mathrm{GS}}
\newcommand{\LS}{\mathrm{LS}}
\newtheorem{theorem}{Theorem}

\newtheorem{mech}{Mechanism}

\newtheorem{definition}{Definition}

\newtheorem{lem}{Lemma}
\newtheorem{remark}{Remark}

\def\tmp#1 #2\relax{#1}
\setbox0=\hbox{$\xdef\intfont{%
    \expandafter\tmp\fontname\textfont3\expandafter\space\space\relax}$}
=\intfont\space at10pt\relax
\setbox0=\hbox{$\textfont3=\tmp \displaystyle \int$}
\xdef\intsize{\the\dimen0}

\def\dividedimen (#1/#2){\expandafter\ignorept\the
   \dimexpr\numexpr\number\dimexpr#1\relax
   *65536/\number\dimexpr#2\relax\relax sp\relax
}
{\lccode`\?=`\p \lccode`\!=`\t  \lowercase{\gdef\ignorept#1?!{#1}}}

\def\flexibleint{\def\fxintL{}\def\fxintU{}\futurelet\next\fxintA}
\def\fxintA{\ifx\next_\expandafter\fxintB\else\expandafter\fxintC\fi}
\def\fxintB_#1{\def\fxintL{#1}\fxintC}
\def\fxintC{\futurelet\next\fxintD}
\def\fxintD{\ifx\next^\expandafter\fxintE\else\expandafter\fxintF\fi}
\def\fxintE^#1{\def\fxintU{#1}\fxintF}
\def\fxintF#1{\begingroup
   \setbox0=\hbox{$\displaystyle{#1}$}%
   \dimen0=\ht0 \advance\dimen0 by\dp0
   \setbox1=\hbox{$\vcenter{\copy0}$}%
   \font\tmp=\intfont\space at\dividedimen(\dimen0/\intsize)pt
   \lower\dimexpr\dp0-\dp1\hbox{%
      $\textfont3=\tmp \displaystyle\int_{\fxintL}^{\fxintU}$}
   \box0
   \endgroup
}

\providecommand{\keywords}[1]
{
  \small	
    \begin{center}\textbf{\textit{Keywords---}} #1\end{center}
}

\title{Differentially private depth functions and their associated medians}
\author{Kelly Ramsay, Shoja'eddin Chenouri}
\date{January 2021}
\begin{document}
\maketitle
\begin{abstract}
In this paper, we investigate the differentially private estimation of data depth functions and their associated medians. 
We introduce several methods for privatizing depth values at a fixed point, and show that for some depth functions, when the depth is computed at an out of sample point, privacy can be gained for free when $n\rightarrow \infty$. 
We also present a method for privately estimating the vector of sample point depth values. 
Additionally, we introduce estimation methods for depth-based medians for both depth functions with low global sensitivity and depth functions with only highly probable, low local sensitivity. 
We provide a general result (Lemma 1) which can be used to prove consistency of an estimator produced by the exponential mechanism, provided the limiting cost function is sufficiently smooth at a unique minimizer. 
We also introduce a general algorithm to privately estimate a minimizer of a cost function which has, with high probability, low local sensitivity. 
This algorithm combines the propose-test-release algorithm with the exponential mechanism. 
An application of this algorithm to generate consistent estimates of the projection depth-based median is presented. 
Thus, for these private depth-based medians, we show that it is possible for privacy to be obtained for free when $n\rightarrow \infty$.
\end{abstract}
\keywords{Differential Privacy, Depth function, Multivariate Median, Propose-test-release}
\section{Introduction}
There is a large body of literature that shows simply removing the identifying information about subjects from a database is not enough to ensure data privacy \citep[see][and the references therein]{Dwork2017}.
Even if only certain summary statistics are released, an adversary can still learn a surprising amount about individuals in a database \citep{Dwork2017}. 
This phenomena is largely due to auxiliary information that is known by the adversary. Given the large amount of information about individuals that is publicly available, it is not infeasible to assume that an adversary already knows some information about the individual they wish to learn about. 
On the contrary, if a statistic is differentially private an adversary cannot learn about the attributes of specific individuals in the original database, regardless of the amount of initial information the adversary possesses. 
This property, coupled with the lack of assumptions on the data itself needed to ensure privacy, accounts for the volume of recent literature on differentially private statistics. 

One part of this literature represents a growing interest in the statistical community in differentially private inference, e.g., \citep[][]{Wasserman2010, Awan2019,Cai2019, Brunel2020}. 
One burgeoning area is the connection between robust statistics and differentially private statistics, first discussed by \cite{Dwork2009}. 
Private M-estimators were studied by several authors \citep{Lei2011, Avella-Medina2019}. 
A connection between private estimators and gross error sensitivity was formalized by \cite{Chaudhuri2012}, who present upper and lower bounds on the convergence of differentially private estimators in relation to their gross error sensitivity. 
The connection between private estimators and gross error sensitivity has been further exploited in order to construct differentially private statistics \citep{Avella-Medina2019}. 
\cite{Brunel2020} greatly expanded the propose-test-release paradigm of \cite{Dwork2009} using the concept of the finite sample breakdown point. 
The same authors use this idea to construct private median estimators with sub-Gaussian errors \citep{Avella-Medina2019a}. 
Our present work is inspired by these recent papers, where we explore the privatization of depth functions, a robust and nonparametric data analysis tool; given the recent success of robust procedures in the private setting, it is worthwhile to develop and study privatized depth functions and associated medians. 

Depth functions give all points in the working space a rating based on how central they are with respect to a dataset; if a point is nested in the dataset, then it will have high depth. 
Depth functions have been widely studied over the last two decades, first used as a means of providing multivariate analogues of certain robust, nonparameteric univariate methods \citep{Zuo2000, Serfling2002} and later as a building block in general inference procedures, e.g., \citep{Cardenas2014, Jeong2016, Chenouri}. 
Some of the inference procedures which can be conducted via depth functions include hypothesis testing \citep{Liu1993,  Serfling2002, Li2004, Chenouri2011}, data visualization \citep{Liu1999}, clustering \citep{Jeong2016}, classification \citep{Lange2014}, outlier detection \citep{Cardenas2014}, change-point problems \citep{Liu1995, Chenouri, RAMSAY2020b} and discriminant analysis \citep{Chakraborti2019}. 
Depth functions provide a large framework for conducting robust, nonparametric inference in multivariate spaces. 
Obtaining private versions of data depth functions will provide immediate private analogues of several of these procedures. 
This means that, through private data depth functions, we may be able to conduct multiple inference procedures without degrading the privacy budget. 
Furthermore, the robustness properties of depth functions, such as a high breakdown point, are well studied and favourable \citep{Romanazzi2001, Chen2002, Zuo2004, Dang2009}, making them a promising direction of study for use in the private setting. 
Here, we take some of the first steps in privatizing depth based inference. 
The contributions are as follows:
\begin{itemize}
    \item We present several approaches for the privatization of sample depth functions, including a discussion of advantages and disadvantages of each approach. 
    \item We present algorithms for the private release of sample depth values of several popular depth functions. 
    These include halfspace depth \citep{Tukey1974}, simplicial depth \citep{liu1990}, IRW depth \citep{RAMSAY201951} and projection depth \citep{zuo2003}. 
    Our algorithms and analysis can also be applied to depth functions with similar characteristics. 
    We present asymptotic results concerning these private, depth value estimates, showing that pointwise, private depth values can be consistently estimated. 
    \item We present algorithms for generating consistent, private depth-based medians, using the exponential mechanism and the propose-test-release framework of \cite{Dwork2009} and \cite{Brunel2020}.  
    \item We extend the propose-test-release algorithm of \cite{Brunel2020} to be used with the exponential mechanism. 
    We present a general algorithm for releasing a private maximizer of an objective function (or minimizer of a cost function) which may have infinite global sensitivity.
    \item We present a general result (Lemma \ref{lem::exp_m}) that can be used to prove weak consistency of private estimators generated from the exponential mechanism, even if the cost function is not necessarily differentiable. 
\end{itemize}
It should be noted that some work has been done surrounding the private computation of halfspace depth regions and the halfspace median \citep{Beimel2019,Gao2020}, mainly from a computational geometry point of view. 
Though \cite{Beimel2019} mentions that the halfspace depth function can be used with the exponential mechanism, they do not study the estimator's properties from a statistical point of view; it is used as a method of finding a point in the convex hull of a set of points. 
Aside from halfspace depth, to the best of our knowledge, no one has studied differentially private versions of other depth functions or depth-based inference. 
\section{Differential Privacy}
Before getting into the fundamentals of differential privacy, it is useful to first introduce some notation. 
Given $x\in \rdd$ we define the $p$-norm as $\norm{x}_p=\left(\sum_{j=1}^d x_j^p\right)^{1/p}$ and given some function $\phi\colon \rdd\rightarrow\re$, we set $\norm{\phi}_{\infty}=\sup_{y\in\rdd} \phi(y)$. 
We represent the data with $\sam{X}{n}=\{X_1,\ldots,X_n\}$ and assume that the data is a random sample of size $n$ such that each observation is in $\rdd$. 
We use $F_n$ to represent the empirical measure determined by $\sam{X}{n}$. 
For a univariate distribution $F$, we use $F^{-1}$ to denote the left continuous quantile function. 
Throughout the paper we define the median of a continuous, univariate distribution by $\med(F)=F^{-1}(1/2)$. 
Both $\med(\sam{X}{n})$ and $\med(F_n)$ are taken to be the usual sample median. 
In other words, $\med(F_n)$ is the usual sample median and not $F_n^{-1}(1/2)$. 
We use $Q_{\sam{X}{n}}$ to represent a measure that depends on the data set $\sam{X}{n}$.  
Differentially private statistics will be denoted with the $\sim$ symbol, e.g., $\widetilde{T}$. 
Given a database $\sam{X}{n}$, we let $\mathcal{D}(\sam{X}{n},k)$ be the set of all databases of size $n$ which differ from $\sam{X}{n}$ by $k$ observations. 

A first essential concept when studying differential privacy is that of a mechanism. 
It has been shown that all differentially private statistics $\widetilde{T}(\sam{X}{n})$ must admit (non-degenerate) measures given the data $Q_{\sam{X}{n}}$. 
This means that given the data, a differentially private statistic (or database) is a random quantity \citep{Dwork2014}. 
We call the procedure that determines $Q_{\sam{X}{n}}$ and then outputs a random draw $\widetilde{T}(\sam{X}{n})$ from $Q_{\sam{X}{n}}$ a mechanism. 
We may also refer to the mechanism by $\widetilde{T}$ with an abuse of notation. 
A second essential concept  for studying differential privacy is that of adjacent databases. 
We say that $\sam{X}{n}$ and $\sam{Y}{n}$ (another random sample of size $n$) are adjacent if they differ by one observation; $\sam{Y}{n}\in \mathcal{D}(\sam{X}{n},1)$. 
In other words, $\sam{X}{n}$ and $\sam{Y}{n}$ are adjacent if their symmetric difference contains one element. 
Equipped with these concepts, we can now define differential privacy: 
\hfill\newpage
\begin{definition}
A mechanism $\widetilde{T}$ is $\epsilon$-differentially private for $\epsilon>0$ if  \begin{equation}
    \frac{Q_{\sam{X}{n}}(B)}{Q_{\sam{Y}{n}}( B)}\leq e^\epsilon
    \label{eqn::dp}
\end{equation} 
holds for all measurable sets $B$ and all pairs of adjacent datasets $\sam{X}{n}$ and $\sam{Y}{n}$.
\label{def::dp}
\end{definition}
\noindent The parameter $\epsilon$ should be small, implying that $$\frac{Q_{\sam{X}{n}}(B)}{Q_{\sam{Y}{n}}( B)}\approx 1,$$
which gives the interpretation that the two measures $Q_{\sam{X}{n}}$ and $Q_{\sam{Y}{n}}$ are almost equivalent. 
To understand this definition, it helps to think of the problem from the adversary's point of view. 
Suppose that we are the adversary and that we have access to all the entries in the database except for one, call it $\theta$, which we are trying to learn about. 
If $\widetilde{T}$ is released, how can we use it to conduct inference about $\theta$? 
Suppose we want to know whether or not $\theta$ belongs to some family of rows $\Theta_0$, i.e., we want to test $$H_0\colon\theta\in \Theta_0 \text{ vs. } H_1\colon \theta\notin \Theta_0.$$ 
To conduct this test, we would then ask two questions: 
\begin{center}
 \textit{How likely was it to observe $\widetilde{T}$ under $H_0$?} and
 \textit{How likely was it to observe $\widetilde{T}$ under $H_1$?}
\end{center}
Differential privacy stipulates that both of these questions have practically the same answer, making it impossible to infer anything about $\theta$ from $\widetilde{T}$. 
Definition \ref{def::dp} implies that if someone in the dataset was replaced, we are just as likely to have seen $\widetilde{T}$ (or some value very close to $\widetilde{T}$ if $Q_{\sam{X}{n}}$ is continuous). 
Another way to interpret the definition is to observe that differential privacy implies that $\mathrm{KL}(Q_{\sam{X}{n}},Q_{\sam{Y}{n}})<\epsilon$, where $\mathrm{KL}$ is the Kullback–Leibler divergence; implying that the distributions are necessarily close. 

One may observe that the inequality \eqref{eqn::dp} must hold for all pairs of adjacent databases and all possible outcomes of the estimator in order for the procedure to be differentially private. 
This inequality is then a worst case restriction, in the sense that \eqref{eqn::dp} must hold for even the worst possible database and the worst possible outcome of the mechanism.  
Definition \ref{def::dp} can be difficult to satisfy because the umbrella of `all databases and mechanism outputs' can include both some extreme databases and extreme mechanism outputs. 
One may wish to relax this definition over unlikely mechanism outputs;
one way to do this is if $B$ is such that $Q_{\sam{X}{n}}( B)$ is very small, then the bound could be allowed to fail. 
This is called approximate differential privacy or $(\epsilon,\delta)$-differential privacy, in which we have
\begin{equation}
    Q_{\sam{X}{n}}( B)\leq e^\epsilon Q_{\sam{Y}{n}}( B)+\delta
      \label{eqn::adp}
\end{equation}
in place of the condition \eqref{eqn::dp}. 
Typically, $\delta << \epsilon$, and $\delta$ can be interpreted as the probability under which the bound is allowed to fail. 
To see this, observe that for $B$ such that $Q_{\sam{X}{n}}( B)<\delta$, \eqref{eqn::adp} holds regardless of $\epsilon$. 
We mention that for remainder of the paper, $\epsilon$ and $\delta$ are always assumed to be positive and that sometimes we may have that the privacy parameters are a function of the sample size, and we indicate this with a subscript $n$; $\epsilon_n,\ \delta_n.$

Central to many private algorithms is the concept of {\it{sensitivity}}. 
Consider some function $T\colon \samplespace\rightarrow \re^{k}$ where $\samplespace$ denotes the Cartesian product of $n$ copies of $\rdd$. 
Usually $T$ represents a statistic or a data driven objective function. 
Sensitivity measures how sensitive $T$ is to exchanging one sample point for another. 
Two important types of sensitivity are local sensitivity and global sensitivity, which are defined as
$$\LS(T;\sam{X}{n})=\sup_{\sam{Y}{n}\in \mathcal{D}(\sam{X}{n},1)}\norm{T(\sam{X}{n})-T(\sam{Y}{n})} \qquad\text{and}\qquad \GS(T)=\sup_{\substack{\sam{X}{n}\in \samplespace,\\ \sam{Y}{n}\in \mathcal{D}(\sam{X}{n},1)}}\norm{T(\sam{X}{n})-T(\sam{Y}{n})}.$$ 
In some cases, it is necessary to use different norms and so we add the subscript $\GS_p$ to indicate global sensitivity computed with respect to the $p$-norm. 

We can now introduce some important building blocks of differentially private algorithms. 
Let $W_1,\ldots,W_k,\ldots $ and $Z_1,\ldots,Z_k,\ldots $ represent a sequence of independent, standard Laplace random variables and a sequence of independent, standard Gaussian random variables, respectively. 
The Laplace and Gaussian mechanisms are essential differentially private mechanisms; they define how much an estimator must be perturbed in order for it to be differentially private. 
\begin{mech}[\cite{Dwork2006}]
Given a statistic $T\colon \samplespace\rightarrow \re^{k}$, the mechanism that outputs
    $$\widetilde{T}(\sam{X}{n})=T(\sam{X}{n})+(W_1,\dots,W_k)\ \frac{\GS_1(T)}{\epsilon},$$
    is $\epsilon$-differentially private.
\label{mech::LM}
\end{mech}
\begin{mech}[\cite{Dwork2006, Dwork2014}]
Given a statistic $T\colon \samplespace\rightarrow \re^{k}$, the mechanism that outputs
    $$\widetilde{T}(\sam{X}{n})=T(\sam{X}{n})+(Z_1,\dots,Z_k)\, \frac{\sqrt{2\log(1.25/\delta)}\,\GS_2(T)}{\epsilon}$$
    is $(\epsilon,\delta)$-differentially private.
    \label{mech::GM}
\end{mech}
This can be improved in strict privacy scenarios \citep{Balle2018}. 
We can also add noise based on smooth sensitivity \citep{Nissim2007}. 
Using smooth sensitivity allows the user to leverage improbable, worst case local sensitivities. 
Often in practice, statistics are computed by maximizing a data driven objective function $\phi_{\sam{X}{n}}(\cdot)$. 
We can privatize such a procedure via the exponential mechanism. 
The exponential mechanism can be defined 
as follows: 
\begin{mech}[\cite{McSherry2007}]
Given the data, consider a function $\phi_{\sam{X}{n}}:\re^{k} \rightarrow \mathbb{R}$ and define the global sensitivity of such a function as $$\GS(\phi)=\sup_{\substack{\sam{X}{n}\in \samplespace,\\ \sam{Y}{n}\in \mathcal{D}(\sam{X}{n},1)}}\norm{\phi_{\sam{X}{n}}-\phi_{\sam{Y}{n}}}_{\infty}.$$ 
Then a random draw from the density $f(v;\phi_{\sam{X}{n}},\epsilon)$ that satisfies
$$f(x;\phi_{\sam{X}{n}},\epsilon)  \propto \exp \left(\frac{\epsilon\  \phi_{\sam{X}{n}}(x) }{2  \GS(\phi)}\right),$$
is an $\epsilon$-differentially private mechanism. 
It is assumed that
$$\int_{\re^{k}} \exp \left(\frac{\epsilon\ \phi_{\sam{X}{n}}(x) }{2  \GS(\phi)}\right)dx<\infty.$$
\label{mech::expMech}
\end{mech}

The factor of 2 can be removed if the normalizing term is independent of the sample. 
All of the mechanisms discussed so far require that the statistic has finite global sensitivity. 
This is a somewhat strict requirement; under the Gaussian model neither the sample mean nor sample median have finite global sensitivity, even when $d=1$. 
The sample median does, however, have low local sensitivity, viz.  $$\LS(\med(\sam{X}{n}))\leq|F_n^{-1}(1/2-1/n)-F_n^{-1}(1/2+1/n)|,$$
when $d=1$.  
Since $1/n\rightarrow 0$, we expect this value to be small (assuming the sample comes from a distribution which is continuous at its median). 

The propose-test-release mechanism, or PTR, can be used to generate private versions of statistics with infinite global sensitivity but highly probable low local sensitivity. 
The propose-test-release idea was introduced by \cite{Dwork2009} but was greatly expanded in the recent paper by \cite{Brunel2020}. 
The PTR algorithm of \cite{Brunel2020} relies on the truncated breakdown point $A_{\eta}$, which is the minimum number of points that must be changed in order to move an estimator by $\eta$:
\begin{equation}
    A_\eta(T;\sam{X}{n})=\min \left\{k\colon \sup_{\sam{Y}{n}\in \mathcal{D}(\sam{X}{n},k)}\norm{T(\sam{X}{n})-T(\sam{Y}{n})}>\eta\right\},
    \label{eqn::tbdp}
\end{equation}
where one recalls that  $\mathcal{D}(\sam{X}{n},k)$ is the set of all samples that differ from $\sam{X}{n}$ by $k$ observations. 
Unlike the traditional breakdown point, the dependence of $A_\eta(T;\sam{X}{n})$ on $\sam{X}{n}$ is important. 
PTR works by proposing a statistic, testing if it is insensitive and then releasing it if it is, in fact, insensitive. 
A private version of $A_\eta(T;\sam{X}{n})$ is used to check the sensitivity. 
\begin{mech}
Given a statistic $T\colon \samplespace\rightarrow \re^{k}$, the mechanism that outputs
\begin{equation}
 \widetilde{T}(\sam{X}{n})=\left\{  \begin{array}{ll}
   \perp  & \text{if } A_\eta(T;\sam{X}{n})+\frac{1}{\epsilon} W_1\leq 1+\frac{\log(2/\delta)}{\epsilon} \\
    T(\sam{X}{n})+\frac{\eta }{\epsilon}W_2 & o.w.
\end{array} \right. 
\label{eqn::ptr_l}
\end{equation}
is $(2\epsilon,\delta)$ differentially private and the statistic
\begin{equation}
 \widetilde{T}(\sam{X}{n})=\left\{  \begin{array}{ll}
   \perp  & \text{if } A_\eta(T;\sam{X}{n})+\frac{\sqrt{2 \log (1.25 / \delta)}}{\epsilon} Z_1\leq 1+\frac{2 \log (1.25 / \delta)}{\epsilon} \\
    T(\sam{X}{n})+\frac{\eta \sqrt{2 \log (1.25 / \delta)}}{\epsilon}Z_2 & o.w.
\end{array} \right. 
\label{eqn::ptr_g}
\end{equation}
is $\left(2 \epsilon, 2 e^{\epsilon} \delta+\delta^{2}\right)$ differentially private. 
\label{mech::PTR}
\end{mech}
The release of $\perp$ means that the dataset was too sensitive for the statistic to be released. 
The goal is to choose a $T$ such that releasing $\perp$ is incredibly unlikely; $A_\eta(T;\sam{X}{n})$ should be large with high probability. 

All of the mechanisms discussed thus far can be combined to produce more sophisticated algorithms, combining two or more mechanisms is called composition. 
One type of composition is computing a function of a differentially private statistic, where the function is defined independently of the data. 
Such statistics are also differentially private. 
It is also true that sums and products of $k$ differentially private procedures each with privacy budget $\epsilon_i$ are $\sum_{i=1}^k \epsilon_i$- differentially private \citep{Dwork2006}. 
This can be improved with advanced composition \citep{Dwork2014}. 
\begin{theorem}[\cite{Dwork2014}]
For given $0<\epsilon<1$ and $\delta'>0$, the composition of $k$ mechanisms which are each $\left(\frac{\epsilon}{2\sqrt{2k\log(1/\delta')}},\delta\right)$-differentially private is $(\epsilon,k\delta+\delta')$-differentially private.
\label{thm::AC}
\end{theorem}

\section{Data Depth}
A data depth function is a robust, nonparametric tool used for a variety of inference procedures in multivariate spaces, as well as in more general spaces. 
A data depth function gives meaning to centrality, order and outlyingness in spaces beyond $\re$.  
Data depth functions do this by giving all points in the working space a rating based on how central the point is in the sample. 
Precisely, we can write multivariate depth functions as $\D\colon \rdd\times F_n \rightarrow \re^+$; given the empirical distribution of a sample $F_n$ and a point in the domain, the depth function assigns a real valued depth to that point. 
Figure \ref{fig:HM}\textcolor{violet}{(a)} shows a sample of 20 points labelled by their depth values, we can see that the points in the center of the data cloud have larger values. 
Note that it is not necessary to restrict the domain of the depth function to points in the sample; we can compute depth values for each point in the sample space. 
The heatmap in Figure \ref{fig:HM}\textcolor{violet}{(a)} gives the depth value for each point in the plot. 
Writing depth functions as functions of the empirical distribution $\D(\cdot ;F_n)$ rather than functions of the sample provides a natural definition for the population depth function $\D(\cdot ;F)$. 
Figure \ref{fig:HM}\textcolor{violet}{(b)} shows the population depth values when $F$ is the two dimensional, standard Gaussian distribution. 

\begin{figure}[t]
\begin{minipage}[c]{0.472\textwidth}
    \centering
    \includegraphics[width=\textwidth]{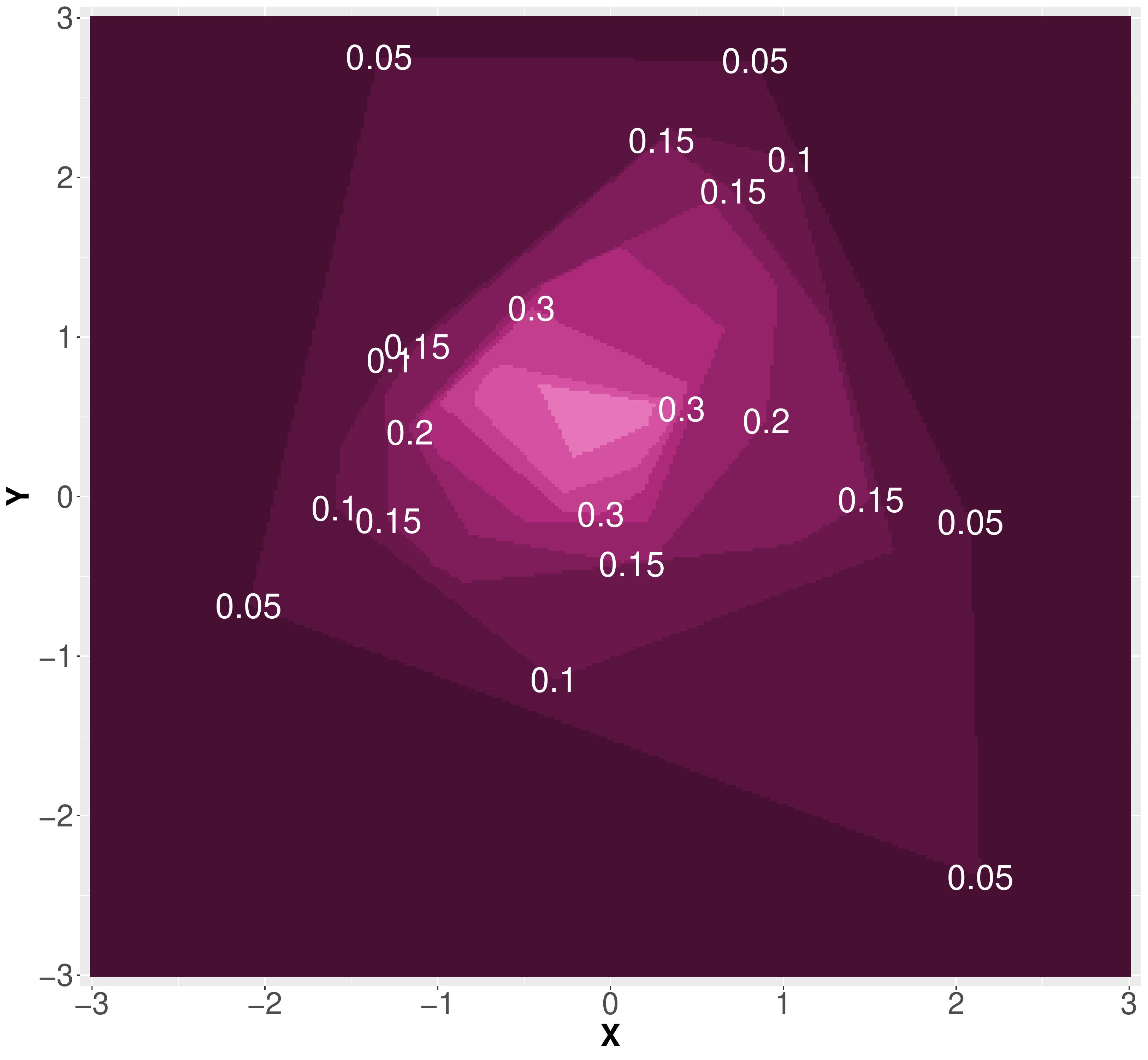}
    \caption*{(a)}
\end{minipage}
\begin{minipage}[c]{0.52\textwidth}
    \centering
    \includegraphics[width=\textwidth]{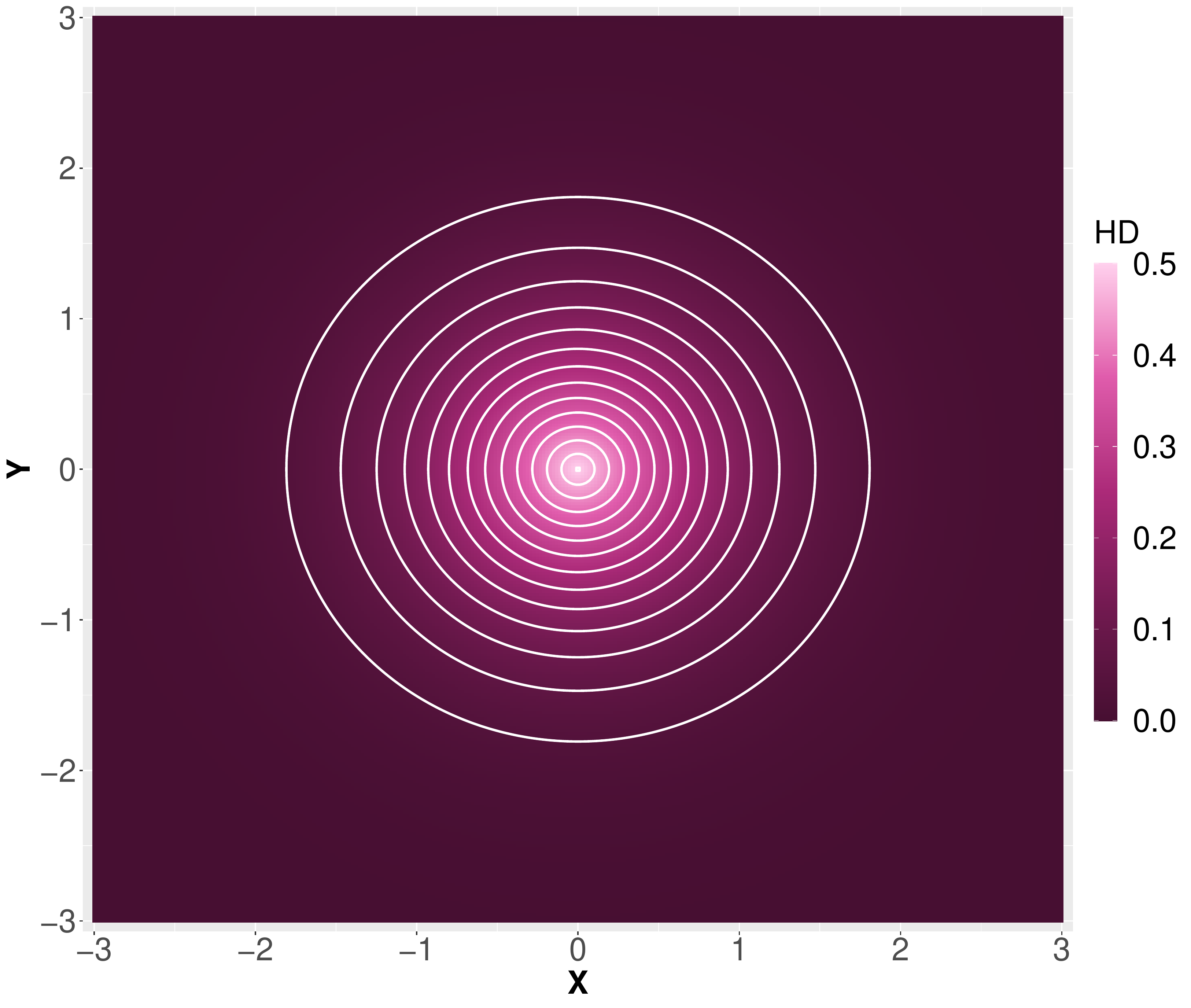}
    \caption*{(b)}
\end{minipage}
\caption[Heatmap of depth values of a sample of 20 points]{(a) Sample halfspace depth values, i.e., $\D(X_i;F_n)$, are displayed in white text. The heatmap of the sample depth function, i.e., $\D(\cdot;F_n)$, is also displayed. This sample is drawn from a standard, two dimensional Gaussian distribution. (b) Theoretical halfspace depth contours for the standard, two dimensional Gaussian distribution.}
\label{fig:HM}
\end{figure}
Depth functions provide an immediate definition of order statistics; observations can be ordered by their depth values. 
However, since the ordering of the sample is center outward, the depth-based order statistics have a different interpretation than univariate order statistics. 
Nevertheless, data depth-based order statistics can be used to define multivariate analogues of many univariate, nonparametric inference procedures.
For example, the definition of the depth-based median is
$$\med(F; \D)=\argmax_{x\in\rdd} \D(x;F).$$
Depth-based medians are generally robust, in the sense that they are not affected by outliers. 
Many depth-based medians have a high breakdown point and favourable properties related to the influence function \citep{Chen2002, Zuo2004}. 
Furthermore, depth-based medians inherent any transformation invariance properties possessed by the depth function. 
We can subsequently define sample depth ranks as 
$$ \R_i= \#\{X_{j }\colon \D(X_{j };F_{n})\leq \D(X_{i};F_{n}),\ j=1,\ldots,n\},$$
which are the building block of various multivariate depth based rank tests \citep{Liu1993, Serfling2002, Chenouri2011, chenouri2012}, as well as providing a method to construct trimmed means \citep{Zuo2002}. 
Depth values can also be used directly in testing procedures \citep{Li2004}. 
Depth functions have also been used for visualization, including the bivariate extension of the boxplot (bagplots) and dd-plots, which allow the analysts to visually compare two samples of any dimension \citep{Liu1999, Li2004}. 
In the same vein of data exploration, we can visualise multivariate distributions through one dimensional curves based on depth values \citep{Liu1999}. 
In the past decade this depth-based inference framework has expanded to include solutions to clustering \citep{Jornsten2004, Baidari2019}, classification \citep{Jornsten2004, Lange2014}, outlier detection \citep{Chen2009OutlierDW, Cardenas2014}, process monitoring \citep{Liu1995}, change-point problems \citep{Chenouri} and discriminant analysis \citep{Chakraborti2019}. 
In summary, depth functions facilitate a framework for robust, nonparametric inference in $\rdd$. 
A major motivating factor for this work is that by privatizing depth functions, we consequentially privatize many of the procedures in this framework. 
This means that private depth values imply access to private procedures for nonparametrically estimating location, performing rank tests, building classifiers and more. 

In their seminal paper \cite{Zuo2000} give a concrete set of mathematical properties which a multivariate depth function should satisfy in order to be considered a {\it{statistical depth function}}. 
These properties include 
\begin{enumerate}
    \item {\it{Affine invariance}}: This implies any depth based analysis is independent of the coordinate system, particularly the scales used to measure the data. 
    \item {\it{Maximality at the center of symmetry}}: If a distribution is symmetric about a point, then surely this point should be regarded as the most central point.
    \item {\it{Decreasing along rays}}: This property ensures that as one moves away from the deepest point, the depth decreases. 
    \item {\it{Vanishing at infinity}}: As a point moves toward infinity along some ray, its depth vanishes. 
\end{enumerate}
A depth function which satisfies these four properties is known as a statistical depth function. 
The last three properties are all related to centrality, where the first is to ensure there is no dependence on the measurement system. 
Not all popular depth functions satisfy all four of these properties, but they typically satisfy most of them. 
Affine invariance, as discussed previously, ensures that the function is not dependent on the coordinate system which, from a practical point of view, means that the measurement scales can be adjusted freely. 
Maximality at center means that if a distribution is symmetric about some point $\theta$, the depth function is maximal at $\theta$. 
Think of the median coinciding with the mean in the univariate case. 
Decreasing along rays means that as one moves along a ray extending from the deepest point, i.e., moves away from the center, the depth value decreases. 
This property can also be replaced with upper semi-continuity. 
Vanishing at infinity means that as the point moves along a ray to infinity, its depth value approaches 0. 
Note that if all four of these properties are not satisfied, it does not necessarily mean that a depth function is invalid or not useful in data analysis; it is merely a limitation to consider. 

Aside from coordinate invariance and centrality, there are other properties that are desirable for a depth function to satisfy. 
We shall list the main ones here
\begin{itemize}
    \item {\it{Robustness:}} A robust depth function implies subsequent inference will be robust, and may make it more amenable to privatization.
    \item {\it{Consistency/Limiting Distribution:}} Consistency for a population depth value and existence of a limiting distribution is useful for developing inference procedures.
    \item {\it{Continuity:}} Continuity can be a building block for consistency and for optimizing the depth function. 
    \item {\it{Computation:}} In order to apply depth-based inference, it is necessary that the depth values are computed quickly. Specifically, being able to compute or approximate the depth values in polynomial time with respect to both $d$ and $n$ is useful.
\end{itemize} 
On top of having these properties, a depth function that is to be used in the private setting should be insensitive. 
In other words, the depth function has low global sensitivity and or highly probable, low local sensitivity. 

We now introduce several depth functions and evaluate their sensitivities. 
The first depth function we will discuss is halfspace depth \citep{Tukey1974}. 
\begin{definition}
[Halfspace depth] Let $S^{d-1}= \{x\in \rdd\colon \ \norm{x}=1\}$ be the set of unit vectors in $\rdd$. Define the halfspace depth $\HD$ of a point $x\in \rdd$ with respect to some distribution $X\sim F$ as
\begin{equation*}
    \HD (x;F)=\inf_{u\in S^{d-1}} \Pr\left(X^\top u\leq x^\top u\right).
\end{equation*}
\label{def::hs}
\end{definition}
\noindent Halfspace depth is the minimum of the projected mass above and below the projection of $x$, over all univariate projections. 
We can interpret the sample depth value of some point $x$ as the minimum normalised, univariate, centre-outward rank of $x$'s projections amongst the samples' projections, over all univariate directions. 
Therefore, if a point is exchanged, all the ranks are shifted by at most one, and the global sensitivity of the unnormalised halfspace depth is 1. 
We get $\GS(\HD)=1/n$, which leads us to conclude that this depth function is relatively insensitive; the global sensitivity is decreasing with respect to the sample size. 
In terms of known properties, halfspace depth is a statistical depth function. 
Its sample depth function is also uniformly consistent for the population depth function \citep{masse2004}. 
Halfspace depth is frequently cited as being computationally complex \cite{Serfling2006}, however, recently an algorithm for computing half-space depth in high dimensions has been proposed \citep{Zuo2019}.

We can replace the minimum in Definition \ref{def::hs} with an average \citep{RAMSAY201951}.
\begin{definition}
[Integrated Rank-Weighted Depth] Define integrated rank-weighted depth as
\begin{equation*}
    \IRW (x;F)=\int_{S^{d-1}} \min\left(\Pr\left(X^\top u\leq x^\top u\right),1-\Pr\left(X^\top u< x^\top u\right)\right) d\nu(u),
\end{equation*}
where $\nu$ is the uniform measure on $S^{d-1}$. 
\label{def::irw}
\end{definition}
It immediately follows from the discussion on the sensitivity of halfspace depth that $\GS(\IRW)=1/n$; this depth function has the interpretation of the average, normalised, univariate centre-outward rank over all projections. 
Therefore, IRW depth is also insensitive. 
Aside from being insensitive, IRW depth vanishes at infinity and is will be maximal at a point of symmetry. 
It is invariant under similarity transformations, which is a weaker form of invariance relative to affine invariance.
It is conjectured that this function also has the decreasing along rays property. 
This depth function is also continuous, and can be approximately computed very quickly \citep{RAMSAY201951}. 
This depth function's sample depth values are also uniformly consistent and asymptotically normal under mild assumptions. 

Another, asymptotically normal depth function is simplicial depth, which was introduced by \cite{Liu1988}. 
\begin{definition} [Simplicial Depth.] 
Suppose that $Y_1, \ldots,Y_{d+1}$ are $\iid$ from $F$. Define simplicial depth as
\begin{equation*}
    \SMD (x;F)= \Pr(x\in \Delta(Y_1, \ldots,Y_{d+1})),
\end{equation*}
where $\Delta(Y_1, \ldots,Y_{d+1})$ is the simplex with vertices $Y_1, \ldots,Y_{d+1}$.
\label{def::smd}
\end{definition}
\noindent We can show that sample simplicial depth has finite global sensitivity. 
Note that 
$$ \SMD (x;F_n)=\frac{1}{\binom{n}{d+1}}\sum_{1<i_1<\ldots <i_{d+1}<n} \ind{X\in \Delta(X_{i_1}, \ldots,X_{i_{d+1}})}.$$
Changing one observation can influence a maximum of $\binom{n-1}{d}$ terms, and each term has a sensitivity of 1. 
It follows that $\GS(\SMD)=(d+1)/n.$ 
Simplicial depth is a statistical depth function if $F$ is angularly symmetric, but fails to satisfy the maximality at center property and decreasing along rays property for some discrete distributions \citep{Zuo2000}. 
Although it is insensitive, this depth function can be difficult to compute in even moderate dimensions ($d>3$). 

The investigation by \cite{Zuo2000} lead to the study of a general and powerful statistical depth function based on outlyingness functions. 
Outlyingness functions $O(\cdot;F)\colon \rdd\rightarrow\re^+$ measure the degree of outlyingness of a point \citep{Donoho1992}. 
A particular version of depth based on outlyingness is projection depth:
\begin{definition}[Projection Depth] 
Given a univariate translation and scale equivariant location measure $\mu$ and a univariate measure of scale $\varsigma$ which is equivariant and translation invariant, we can define projected outlyingness as 
$$O(x;F;\mu, \varsigma)=\sup_{u\in S^{d-1}}\frac{\left|u^\top x-\mu(F_u)\right|}{\varsigma(F_u)}$$
and thus projection depth as,
$$\PD(x;F;\mu, \varsigma)= \frac{1}{1+O(x;F;\mu, \varsigma)}.$$
\label{dfn::pd}
\end{definition}
Typically, $\mu$ and $\varsigma$ refer to the median and median absolute deviation, but properties have been investigated for general $\mu$ and $\varsigma$. 
One idea is to design $\mu$ and $\varsigma$ such that $O(x;F_n)$ has low global sensitivity, but that is left to later work. 
Here, we will use either
$$O_1(x;F_n)\coloneqq O(x;F_n;\med,\mad)=\sup _{u\in S^{d-1}} \frac{\left|u^\top x-\med\left( \sam{X}{n}^\top u\right)\right|}{\mad\left( \sam{X}{n}^\top u\right)}$$
or
$$O_2(x;F_n)\coloneqq O(x;F_n;\med,\iqr)=\sup _{u\in S^{d-1}} \frac{\left|u^\top x-\med\left( \sam{X}{n}^\top u\right)\right|}{\iqr\left( \sam{X}{n}^\top u\right)}.$$
The global sensitivities of $O_1$, $O_2$ are unbounded, implying that the global sensitivity of $\PD$ is equal to 1. 
Seeing as the range of projection depth is $[0,1)$, a global sensitivity of 1 is high. 
Both $O_1$, $O_2$ have bounded local sensitivities, making projection depth a good candidate for the propose-test-release procedure. 
Note that we use a slight abuse of notation, where $\sam{X}{n}^\top u$ refers to the sample $\lbrace X_1^\top u,\ldots, X_n^\top u\rbrace$. 
We may also refer to the empirical distribution implied by this sample as $F_{n,u}$. 
A thorough investigation of the properties of projection depth was done in the successive papers \citep{zuo2003, Zuo2004}. 
As a result of these papers, it has been shown that projection depth is a statistical depth function, sample projection depth values have a limiting distribution and these sample depth values are robust against outliers. 

\section{Private Data Depth}
There are several ways in which we could approach privatizing depth functions. 
A natural and easy way to do this is to start with a differentially private estimate of the distribution of the data $\widetilde{F}_n$ and use $\D(x,\widetilde{F}_n)$, which is differentially private. 
Computing $\widetilde{F}_n$ relies on existing methods for generating private multidimensional empirical distribution functions. 
Methods based on differentially private estimates of the distribution fail to take advantage of any robustness properties of depth functions; they do not leverage the low sensitivity of the depth function itself. 
This method also does not give a method for computing the sample depth values $\D(X_1,F_n)$, since $\D(X_1,\widetilde{F}_n)$ is not differentially private. 
Computing the sample depth values is often included in depth-based inference \citep[see, e.g.,][]{Li2004, Lange2014}. 
In this paper, we aim to study the advantages of the robustness properties of depth functions in the differentially private setting and so we forgo study of $\D(x,\widetilde{F}_n)$. 

If the global sensitivity of $\D$ is finite, then an obvious private estimate is 
$$\widetilde{\D}(x;F_n)=\D(x;F_n)+V_{\delta,\epsilon}\GS(\D),$$
where $V_{\delta,\epsilon}$ is independent noise, from either the Laplace or Gaussian distribution with the scale parameter calibrated to ensure differential privacy. 
If a depth function has high global sensitivity, then, given the robustness properties of most depth functions, it makes sense to apply a propose-test-release algorithm. 

We can also produce a more direct privatized estimate of $\D(\cdot,F)$ based on the sample, such as has been done with histogram bins \citep{Wasserman2010}. 
For example, many depth functions are defined based on functions of projections: $h(\cdot ;\sam{X}{n}^\top u)$, $u\in \mathscr{U}_n$ where $\mathscr{U}_n$ is some set of directions, i.e., $\mathscr{U}_n\subset S^{d-1}$. 
We could then produce private versions of $\sam{X}{n}^\top u$ or private versions of $h(\cdot ;\sam{X}{n}^\top u)$, if $h$ is insensitive. 
The advantage of this approach would be that the entire depth function could be privatized at once, including the sample depth values. 
In the same vein, recalling that $\nu$ is the uniform measure on $S^{d-1}$, there exists an induced measure, $\gamma_{x,{F_n}}(A)=\nu(h^{-1}(A;x,F_n))$ on the Borel sets of the range of the depth function, i.e., $A\in \mathscr{B}$. 
If $\gamma_{x,{F_n}}$ is insensitive, then we can construct a differentially private estimator based on random draws from $\gamma_{x,{F_n}}$. 
This approach is somewhat complicated, and could be tedious if $\gamma_{x,{F_n}}$ is difficult to sample from. 
Since $\gamma_{x,{F_n}}$ depends on $x$, we would have to set up a different sampler for each $x$ at which we want to compute a depth value. 
We leave these projection type approaches for future research.

From the discussion above, it is clear that a key question is at which points would we like to estimate depth values? 
Algorithms which estimate the depth of a single point are of course of interest; they can be composed to compute depth values at several points privately. 
Additionally, simple algorithms to compute the depth value of a single point can be used as building blocks for private versions of depth based inference procedures. 
As mentioned previously, it is also of interest to compute the depth values of the sample points:
$$\widehat{\mathbf{D}}(F_n)\coloneqq ( \D(X_1;F_n), \D(X_2;F_n),\ldots , \D(X_n;F_n)) . $$
Since $X_i$ appears in both arguments, the sensitivity of $\D(X_i;F_n)$ is larger than that of $\D(x;F_n)$ $x\notin \sam{X}{n}$. 
We investigate private methods of estimating the vector of sample depth values. 
A further question is whether or not we can estimate several depth values from different samples simultaneously, e.g., for use in depth-based clustering. 
To elaborate, if $\sam{X}{n}$ contains the samples for $J$ groups, then $\sam{X}{n}=\sam{X}{n}^1\cup \ldots\cup \sam{X}{n}^J$. 
For example, if we privatize the one dimensional projections of the entire sample $\sam{X}{n}^\top u$ we can then compute the depth of each point in $\sam{X}{n}$ with respect to each group $\sam{X}{n}^j$.

An important question is how well do the privatized inference procedures perform when compared to their non-private counterparts. 
Do the privatized depth values converge to their non-private counterparts? 
If so, what is the rate of convergence? 
Does this private estimate have a limiting distribution? 
If so, is the limiting distribution different from that of the non-private limiting distribution? 
We investigate some of these questions in the next section. 

\section{Algorithms for Private Depth Values}\label{sec::pdv}
As mentioned previously, for depth functions with finite global sensitivity, we can make use of the Gaussian and Laplace mechanisms. 
\begin{mech}
For $x$ given independently of the data, the following estimators
$$\widetilde{\D}_1(x;F_n)=\D(x;F_n)+W\frac{\GS_1(\D)}{\epsilon}\qquad\text{and}\qquad \widetilde{\D}_2(x;F_n)=\D(x;F_n)+Z \frac{\GS_2(\D)\ \sqrt{2\log(1.25/\delta)}}{\epsilon}$$
are $\epsilon$-differentially private and $(\epsilon,\delta)$-differentially private, respectively.
\label{mech::sf}
\end{mech}
\noindent 
The fact that these mechanisms are differentially private follow from the differential privacy of Mechanisms \ref{mech::LM} and \ref{mech::GM}. 
The following results are immediate:
\begin{theorem}
For a given depth function $\D$ and a fixed $x\in \rdd$, suppose that $\sqrt{n}(\widetilde{\D}_\ell(x;F_n)-\D(x;F))\cond \mathcal{V}_{\D}(x),$ where $\cond$ denotes convergence in distribution. 
Suppose we can write $\GS(\D)= C(\D)/n$ where $C(\D)$ does not depend on $n$. 
Let $r >0$ and $\ell=1$ or $\ell=2$. 
For depth values generated under Mechanism \ref{mech::sf}, the following holds
\begin{enumerate}
    \item  For $\delta_n=o(n^{-k})$ and $\epsilon_n=O(n^{-1+r})$, $\widetilde{\D}_\ell(x;F_n)\conp \D(x;F),$
    where $\conp$ denotes convergence in probability. 
    \item  For $\delta_n=o(n^{-k})$ and $\epsilon_n=O(n^{-1/2+r})$, $\sqrt{n}(\widetilde{\D}_\ell(x;F_n)-\D(x;F))\cond \mathcal{V}_{\D}(x).$
\end{enumerate}
\label{thm::gshs}
\end{theorem}
It should be noted that choosing $\delta=o(n^{-k})$ and $\epsilon=O(n^{-1/2+r})$ maintains a reasonable level of privacy. 
For example, choosing $\epsilon\in O(1)$ and $\delta<1/n$ is ``the most-permissive setting under which $(\epsilon,\delta)$-differential privacy is a nontrivial guarantee'' \citep{Cai2019}. 
From Theorem \ref{thm::gshs} we can conclude that for large samples and small privacy parameters, depth value estimates generated via Mechanism \ref{mech::sf} are minimally affected by privatization. 

What if we want to calculate the depth value of a sample point? 
How we can estimate the vector of depth values at the sample points, i.e.,  
\begin{equation}
    \widehat{\mathbf{D}}(F_n)=( \D(X_1;F_n), \D(X_2;F_n),\ldots , \D(X_n;F_n)) 
    \label{eqn::vecd}
\end{equation}
privately?
The sample values now appear in both arguments of $\D$ and so we must do a bit more work to compute the global sensitivity. 
First we look at halfspace and IRW depth. 
Consider one set of projections $\sam{X}{n}^\top u$ and their corresponding empirical distribution $F_{n,u}$. 
We want to compute the sensitivity of the vector $\mathbf{R}_{n,u}=(R_{1,u},\ldots,R_{n,u})$, with
$$R_{i,u}=\min\{F_{n,u}(X_{i}^\top u),1-F_{n,u}(X_{i}^\top u-) \},$$
and $F(x-)=P(X<x)$. 
If we change one observation, if $n$ is odd, then at most $n-1$ ranks can be changed by at most 1, and thus, at most $n-1$ values of $\mathbf{R}_{n,u}$ would change by at most $1/n$. 
In the even case, it is easily seen that at most $n-2$ values of $\mathbf{R}_{n,u}$ will change by at most $1/n$. 
Alternatively, we can change one depth value by $n^{-1}(\lfloor(n+1)/2\rfloor-1)$ and $\lfloor(n+1)/2\rfloor-1$ values by $1/n$. 
This gives that 
$$\GS_1(\mathbf{R}_{n,u})=\frac{2}{n}\left(\left\lfloor\frac{n+1}{2}\right\rfloor-1\right)\approx 1,$$
and that
$$\GS_2(\mathbf{R}_{n,u})=\frac{1}{n}\sqrt{\left(\left\lfloor\frac{n+1}{2}\right\rfloor-1\right)^2+\left\lfloor\frac{n+1}{2}\right\rfloor-1}\approx 1. $$
Since averaging or taking the supremum over such $\mathbf{R}_{n,u}$ does not affect these sensitivities, it follows that for halfspace depth and IRW depth $\GS_\ell(\mathbf{D})=\GS_\ell(\mathbf{R}_{n,u}).$ 
Concerning simplicial depth, with respect to some adjacent dataset, the depth values of the unchanged points can each change by at most $(d+1)/n$. 
For the point that is different, we can bound the sensitivity above by $1-(d+1)/n$. 
It follows that $\GS_1(\mathbf{SD})\leq 1$ and $\GS_2(\mathbf{SD})\leq\sqrt{(1-(d+1)/n)^2+((d+1)/n)^2}\approx 1$,  where $\mathbf{SD}$ is the vector $\mathbf{D}$ in \eqref{eqn::vecd} with $\D=\SMD$. 
In summary, the global sensitivities of the vector of sample depth values for halfspace, IRW and simplicial depth are all close to 1. 
This means that when computing $n$ depth values from halfspace, IRW or simplicial depth privately, we pay approximately the same privacy budget regardless of whether or not the points at which the depth is being computed are in the sample. 
In other words, we do not use any extra privacy budget for the fact that we are computing the depth at the sample values, provided we are computing $n$ of them. 
We can then use the following mechanism to estimate the vector of depth values:
\begin{mech}
The following estimators for the vector of depth values of the sample points 
$$\widetilde{\mathbf{D}}_1(F_n)=\mathbf{D}(F_n)+(W_1,\ldots,W_n) \frac{\GS_1(\mathbf{D})}{\epsilon}\qquad\text{and}\qquad \widetilde{\mathbf{D}}_2(F_n)=\mathbf{D}(F_n)+(Z_1,\ldots,Z_n)\frac{\GS_2(\mathbf{D})\ \sqrt{2\log(1.25/\delta)}}{\epsilon}$$
are $\epsilon$-differentially private and $(\epsilon,\delta)$-differentially private, respectively.
\label{mech::sfvec}
\end{mech}
The fact that these mechanisms are differentially private follow from the differential privacy of Mechanisms \ref{mech::LM} and \ref{mech::GM}.
For the full vector of sample depth values we do not get privacy for free in the limit. 
For many depth functions, certainly for halfspace depth, IRW depth and simplicial depth, we have that 
$$\sup_x|\D(x;F_n)-\D(x;F)|=O_p(n^{-1/2}),$$
which gives that 
$$\norm{\mathbf{D}(F_n)-\mathbf{D}(F)}\leq \sqrt{n\left(\sup_x|\D(x;F_n)-\D(x;F)|\right)^2} =O_p(1).$$
Then, from the triangle inequality we have that
{\footnotesize
\begin{align}
\label{eqn::dep_vec}
    \norm{\widetilde{\mathbf{D}}_1(F_n)-\mathbf{D}(F)}&\leq \norm{\widetilde{\mathbf{D}}_1(F_n)-\mathbf{D}(F_n)}+\norm{\mathbf{D}(F_n)-\mathbf{D}(F)}=\norm{(W_1,\ldots,W_n) \GS(\mathbf{D})/\epsilon}+O_p(1)=O_p(n^{1/2}). 
\end{align}}
For $\widetilde{\mathbf{D}}_2$ with $\delta\propto n^{-k}$, we have that
$$\norm{\widetilde{\mathbf{D}}_2(F_n)-\mathbf{D}(F)}\leq O_p(n^{1/2}\log^{1/2}{n}).$$
The level of noise is greater than that of the sampling error for both of these private estimates of the vector of depth values at the sample points. 
It is important that this be accounted for when developing inference procedures based on these privatized estimates. 
This result is somewhat intuitive; these vectors reveal more information about the population as $n$ grows, which differs markedly from the single depth value case, where the amount of information received is fixed in $n$. 
In fact, for large $n$ the vector of depth values at the sample points contains a significant amount of information about $F$; the population depth function can, under certain conditions, characterize the distribution of $F$ \citep[see][and the references therein]{Nagy2018}. 
To release so much information about the population privately, we need to inject greater than negligible noise. 

We now turn our attention to a depth function with high global sensitivity: projection depth. 
For projection depth, we would like to generate private outlyingness values, which have unbounded sensitivity. 
Note that $\med,\ \mad,\ \iqr$ are all robust statistics, in the sense that they are not perturbed by extreme data points. 
This implies that $O_1$ and $O_2$ have an unlikely chance of worst case sensitivity, which would make projection depth a good candidate for the propose-test-release framework \citep{Dwork2009, Brunel2020}. 
Suppose that $\iqr(F_{n,u})\approx 1$ for all $u$. 
If $$\eta\gtrapprox\max(F^{-1}_{n,u}(1/2)-F^{-1}_{n,u}(1/2-1/n),F^{-1}_{n,u}(1/2+1/n)-F^{-1}_{n,u}(1/2))$$
for all $u$ then $A_\eta\approx\floor{n/2}-1$, which means it is very unlikely that Mechanism \ref{mech::PTR} will return $\perp$. 
This is the basis for introducing the following mechanism:
\begin{mech} 
For $\ell=1$ or $\ell=2$, define privatized projection depth as 
$$\widetilde{\PD}_\ell(x;F_n)=\frac{1}{1+\widetilde{O}_\ell(x;F_n)},$$
where
\begin{equation*}
\widetilde{O}_\ell(x;F_n)=\left\{  \begin{array}{ll}
   \perp  & \text{if } A_\eta(O_\ell(x;F_n);\sam{X}{n})+\frac{a_\delta}{\epsilon} V_1\leq 1+\frac{b_\delta}{\epsilon} \\
    O_\ell(x;F_n)+\frac{\eta a_\delta}{\epsilon}V_2 & o.w.
\end{array} \right.
\end{equation*}
where $a_\delta,\ b_\delta$, and $V_j$ are either as in \eqref{eqn::ptr_l} or as in \eqref{eqn::ptr_g} and the corresponding level of privacy is as in Mechanism \ref{mech::PTR}. 
\label{mech::PTR_PD}
\end{mech}
\noindent One thing to note is that Mechanism \ref{mech::PTR_PD} can be used to estimate both in sample and out of sample points. 
We can actually show that this algorithm is consistent for the population depth values when using $O_2$ as the outlyingness measure. 
\begin{theorem}
Let $\xi_{p,u}$ be the $p^{th}$ quantile of $F_u$. 
Suppose that for all $h>0$, and all $u\in S^{d-1}$
$$|F_u(\xi_{p,u}+h)-F_u(\xi_{p,u})|=M|h|^q(1+O(|h|^{q/2})\,)\text{ with }M>0,\ q>0,$$
for $p=1/4,\ 1/2,\ 3/4$. 
Suppose that $\sup_u \xi_{p,u}<\infty$ for $p=1/4,\ 3/4$. 
For $\eta\propto\frac{\log n}{n^{3/4-r}}$ with $r>0$, $\delta_n=O(n^{-k})$ and $\frac{n^{1/4}(\log\log n)^{3/4}}{\log \delta_n}\epsilon_n\rightarrow\infty$ it holds that
$$|\widetilde{\PD}_2(x;F_n)-\PD(x;F;\med,\iqr)|\conp 0.$$
\label{thm::pd_con}
\end{theorem}
Theorem \ref{thm::pd_con} shows we can choose both $\epsilon_n$ and $\eta_n$ decreasing in $n$ and still maintain a consistent estimator. 
In fact, $\eta_n$ can be chosen to be quite small relative to the size of the sample. 
Under the Laplace version of Mechanism \ref{mech::PTR_PD}, if the statistic is released, recall that the scale parameter is proportional to $\eta\epsilon^{-1}$. 
This means that Mechanism \ref{mech::PTR_PD} injects a smaller amount of noise into the depth value than would Mechanism \ref{mech::sf}, however, this is paid for with the privacy budget of $(2\epsilon,\delta)$, rather than a budget of just $\epsilon$. 
The difference in noise is on the order of $(\log\log n)^{-3/4+r'},\ r'>0$, therefore, in smaller samples the gain would be negligible. 

In terms of computing this estimator, the difficulty lies in computing $A_\eta$ for a given dataset. 
This is non-trivial for projection depth, as the ratio of estimators makes the computation difficult. 
We can approximate the depth value by computing 
$$\widehat{O}_\ell=\max_{u\in U_1,\ldots,U_m} \frac{|x^\top u-\med(\sam{X}{n}^\top u )|}{\varsigma_\ell(\sam{X}{n}^\top u )},$$
instead of computing 
$$O_\ell=\sup_{u\in S^{d-1}} \frac{|x^\top u-\med(\sam{X}{n}^\top u )|}{\varsigma_\ell(\sam{X}{n}^\top u )},$$ 
where $U_1,\ldots,U_m$ are sampled uniformly from $S^{d-1}$. 
Then, we can compute the truncated breakdown point of each $$\widehat{O}^{U_j}_\ell(x)=\frac{|x^\top U_j-\med(\sam{X}{n}^\top U_j )|}{\varsigma_\ell(\sam{X}{n}^\top U_j)},$$
to construct an approximation of the breakdown value $A_\eta$. 
It may also be possible to compute this estimator exactly using techniques from computational geometry \citep[see, e.g.,][]{Liu2014}. 

Suppose that $\sam{Y}{n}\in \mathcal{D}(\sam{X}{n},k^*)$, where $k^*=1+\frac{b_{\delta}}{\epsilon}-\frac{a_{\delta}}{\epsilon} V_{1}$. 
We first an algorithm to check if  $A_\eta\left(\widehat{O}^{u}_2(x);\sam{X}{n}\right) > k^*$.  
To this end, note that $$|x-\med(\sam{Y}{n}^\top u)|\leq \max\{|x-F_{n,u}^{-1}(1/2+k^*/n)|,|x-F_{n,u}^{-1}(1/2-k^*/n)|\}$$
and that
$$|x-\med(\sam{Y}{n}^\top u)|\geq \min\{|x-F_{n,u}^{-1}(1/2+k^*/n)|,|x-F_{n,u}^{-1}(1/2-k^*/n)|,|x-m_1(u)|,|x-m_2(u)|\},$$
with $m_1(u)$ being the median of a dataset the same as $\sam{X}{n}^\top u$, except that the smallest $k^*$ observations of $\sam{X}{n}^\top u$ are replaced with $x^\top u$ and $m_2(u)$ being the same as $m_1(u)$, except instead the largest $k^*$ observations of $\sam{X}{n}^\top u$ are replaced. 
Define
$$\mathcal{B}=\{F_{n,u}^{-1}(3/4+k_1/n)-F_{n,u}^{-1}(1/4+k_2/n)\colon -k^*\leq k_1,k_2\leq k^*,\ |k_1|+|k_2|=k^*\},$$
with
$$\min \mathcal{B}\leq \iqr(\sam{Y}{n}^\top u)\leq \max\mathcal{B}.$$
We can summarise these bounds by letting
\begin{align*}
    \mathrm{up}(\med,u)&=\max\{|x-F_{n,u}^{-1}(1/2+k^*/n)|,|x-F_{n,u}^{-1}(1/2-k^*/n)|\},\\
    \mathrm{lo}(\med,u)&=\min\{|x-F_{n,u}^{-1}(1/2+k^*/n)|,|x-F_{n,u}^{-1}(1/2-k^*/n)|,|x-m_1(u)|,|x-m_2(u)|\},\\
    \mathrm{lo}(\iqr,u)&=\min \mathcal{B},\\
    \mathrm{up}(\iqr,u)&=\max\mathcal{B}.
\end{align*}
Using this notation, we can write
$$\widehat{O}^{u}_\ell(x)\in \left[\frac{\mathrm{lo}(\med,u)}{\mathrm{up}(\iqr,u)},\frac{\mathrm{up}(\med,u)}{\mathrm{lo}(\iqr,u)}\right]=\left[ \mathrm{lo}(\widehat{O}^{u}_\ell(x)), \mathrm{up}(\widehat{O}^{u}_\ell(x))\right]$$
and we can check if
\begin{equation}
    \max\{\widehat{O}^{u}_\ell(x)-\mathrm{lo}(\widehat{O}^{u}_\ell(x)),\mathrm{up}(\widehat{O}^{u}_\ell(x))-\widehat{O}^{u}_\ell(x)\}<\eta
    \label{eqn:bda}
\end{equation}
Then, if \eqref{eqn:bda} holds for any $u\in \{U_1,\ldots,U_m\}$, we must have that $A_\eta(\widehat{O}_2(x;F_n),\sam{X}{n}) > k^*$, 
which gives a lower bound on the truncated breakdown point. 
This lower bound can be used when implementing Mechanism \ref{mech::PTR_PD}, in the `test' portion of the algorithm. 

The methods used to construct private depth values discussed in this section can be used to privatize inference procedures based solely on functions of sample depth values.  
For example, a common way to compare scale between two multivariate samples, say $\sam{X}{n_1}$ and $\sam{Y}{n_2}$, is to compute the sample depth values with respect to the empirical distribution of the pooled sample $\sam{X}{n_1}\cup \sam{Y}{n_2}$ \citep{Li2004, Chenouri2011}. 
We can denote this empirical distribution by $G_{n_1+n_2}$. 
Private depth-based ranks could then be defined as
$$\widetilde{\R}_{ji}=\#\{ X_{k\ell}\colon \ \widetilde{\D}(X_{k\ell};G_{n_1+n_2})\leq\widetilde{\D}(X_{ji};G_{n_1+n_2}),\ j,k\in\{1,2\}\ i\in\{1,\ldots,n_j\},\ \ell\in\{1,\ldots,n_k\} \},$$
where $X_{ji}$ is the $i^{th}$ observation from sample $j$. 
We can use these ranks to privately test for a difference in scale between the two groups with the rank sum test statistic, viz. 
$$\widetilde{T}(\sam{X}{n_1}\cup\sam{Y}{n_2})=\sum_{i=1}^{n_1} \widetilde{\R}_{ji}.$$
The distribution of such a statistic remains the same under the null hypothesis, and \eqref{eqn::dep_vec} can be used to assess its performance under the alternative hypothesis. 
It is clear that the power of the test will be reduced, as the noise biases the statistic toward failing to reject the null hypothesis. 
We can also take a similar approach in multivariate, covariance change-point models \citep{Chenouri, RAMSAY2020b}. 
The algorithms of this section cannot be used to compute private depth-based medians, i.e., private maximizers of the depth functions, and so we investigate algorithms to compute depth-based medians in the next section.
\section{Private Multivariate Medians}\label{sec::pm}
For depth functions with finite global sensitivity, it is natural to estimate the depth-based median using the exponential mechanism (Mechanism \ref{mech::expMech}).  
As such, we could generate an observation from
$$f(v;F_n)\propto \exp\left(\dfrac{\epsilon}{2\GS(\D)}\D(v;F_n)\right),$$
to be used as a private estimate of the $\D$-based median. 
One issue is that this density is not necessarily valid. 
For example, 
$$f(v;F_n)= \frac{\exp\left(\dfrac{\epsilon}{2\GS(\HD)}\HD(v;F_n)\right)}{\int_{\rdd}\exp\left(\dfrac{\epsilon}{2\GS(\HD)}\HD(v;F_n)\right)dv},$$
is not a valid density, since $\int_{\rdd}\exp\left(\dfrac{\epsilon}{2\GS(\D)}\D(v;F_n)\right)dv=\infty$. 
To see this, note that 
$$1<\exp\left(\dfrac{\epsilon}{2\GS(\HD)}\HD(v;F_n)\right)<\infty$$
and so even if we transform this to 
\begin{align*}
\exp\left(-\dfrac{\epsilon}{2\GS(\HD)}(\alpha-\HD(v;F_n))\right),
\end{align*} 
it is still bounded below for any $\alpha$. This implies that 
$$\int_{\rdd}\exp\left(\dfrac{\epsilon}{2\GS(\HD)}\HD(v;F_n)\right)dv=\infty.$$
Similar results follow for the remaining depth functions, since they all have a range that lies in a positive, bounded interval. 
If the data for which we would like to estimate the median is within some compact set $B$, then we can easily reduce the range of the estimator to $B$ and the density 
$$f(v;F_n)= \frac{\exp\left(\dfrac{\epsilon}{2\GS(\HD)}\HD(v;F_n)\right)\ind{v\in B}}{\int_{B}\exp\left(\dfrac{\epsilon}{2\GS(\HD)}\HD(v;F_n)\right)dv},$$
is valid. 
If there is no clear set $B$ in which the median will lie then we propose a Bayesian inspired approach, and recommend using a prior $\pi(v)$ on the median such that 
$$f(v;F_n)= \frac{\exp\left(\dfrac{\epsilon}{2\GS(\HD)}\HD(v;F_n)\right)\pi(v)}{\int_{\rdd}\exp\left(\dfrac{\epsilon}{2\GS(\HD)}\HD(v;F_n)\right)\pi(v)dv},$$
is a valid density. 
Seeing as $\ind{v\in B}$ normalized by $\int_B dv$ is the density of the uniform distribution over $B$, it is a special case of a prior and we can summarise this procedure as follows:
\begin{mech}
Suppose that $\GS(\D)=C(\D)/n$. Suppose also that $\pi(v)$ is a density chosen independently of the data. 
Provided  
\begin{equation}
    f(v;F_n)=\frac{\exp\left(\dfrac{n\epsilon}{2C(\D)}\D(v;F_n)\right)\pi(v)}{\int_{\mathbb{R}^d} \exp\left(\dfrac{n\epsilon}{2C(\D)}\D(v;F_n)\right)\pi(v)dv},
    \label{eqn::dens}
\end{equation}
is a valid Lebesgue density, a random draw from $f(v;F_n)$ is an $\epsilon$-differentially private estimate of the $\D$-based median of $\sam{X}{n}$.
\label{mech::exp_med}
\end{mech} 
It is imperative that this prior is chosen independently of the data or the privacy of the procedure will be violated. 
For any depth function whose range is a bounded interval, it is easy to see that \eqref{eqn::dens} is a valid density. 
Suppose that the range of $\D$ is [0,1], then the following inequality holds
\begin{align*}
   1= \int_{\mathbb{R}^d} \pi(v)dv\leq \int_{\mathbb{R}^d} \exp\left(\dfrac{n\epsilon}{2C(\D)}\D(v;F_n)\right)\pi(v)dv&\leq \int_{\mathbb{R}^d} \exp\left(\dfrac{n\epsilon}{2C(\D)}\right)\pi(v)dv=\exp\left(\dfrac{n\epsilon}{2C(\D)}\right).
\end{align*}
Some asymptotic properties of the exponential mechanism have been investigated by \cite{Awan2019}, but their result requires that the cost function is twice differentiable and convex. 
Depth functions do not typically satisfy these requirements. 
The following lemma is useful for proving some asymptotic results related to the exponential mechanism, when the cost function is not necessarily differentiable, but smooth at the limiting minimizer
\begin{lem}
Let $\mathbf{0}$ be the zero vector in $\rdd$ and $\pi(v)$ be a density on $\rdd$. 
Suppose that $\phi_n(\omega, v)\colon \Omega\times \mathbb{R}^d\rightarrow\mathbb{R}^+$ is a sequence of random functions on the probability space $(\Omega,\mathscr{A},P)$. 
Assume that 
\begin{enumerate}
\item $\lambda_n=Cn^r$ for some $C>0$.
\item For $s>r$, $\norm{\phi_n(\omega, \cdot)-\phi(\omega, \cdot)}_\infty=\sup_{v\in\rdd}\norm{\phi_n(\omega, v)-\phi(\omega, v)}=o_p(n^{-s})$.
\item For some $\alpha >0$, $\phi(\omega, v)$ is $\alpha$-H\"{o}lder continuous in $v$ in a neighborhood around $\mathbf{0}$ $P-$almost surely. 
This means that $|\phi(\omega,v)-\phi_(\omega,\mathbf{0})|\leq C_1 \norm{v}^\alpha$ for some constant $C_1$. 
\item  $\phi(\omega, v)=0$ if and only if $v=\mathbf{0}$ $P-$almost surely; $\phi$ is uniquely minimized at $v=\mathbf{0}$ $P-$almost surely. 
\item $\pi(v)$ is a bounded Lebesgue density which is positive in some neighborhood around $\mathbf{0}$. 
\end{enumerate}
Let $V_n$ be a sequence of random vectors whose distribution on $\mathbb{R}^d$ is given by the measure
$$Q_n(A)= \int_\Omega\int_A\dfrac{e^{-\lambda_n\phi_n(\omega,v)}\pi(v)dv}{\int_{\mathbb{R}^d}  e^{-\lambda_n\phi_n(\omega,v)}\pi(v)dv}dP,$$
for $A\in \mathscr{B}(\mathbb{R}^d)$. 
Then $V_n\conp \mathbf{0}$. 
\label{lem::exp_m}
\end{lem}
Lemma \ref{lem::exp_m} may be applied outside the context of depth functions, and can be used to prove weak consistency of an estimator based on the exponential mechanism. 
When applying Lemma \ref{lem::exp_m}, the sequence $\lambda_n$ should be replaced by the ratio of the privacy parameter and the global sensitivity of the cost function. 
This lemma shows that smoother, insensitive cost functions will allow the estimator to be consistent for smaller privacy budgets, provided that the prior is positive in a region around the maximizer. 
Additionally, if $e^{-\lambda_n\phi_n(\omega,v)}$ is integrable, then we can let $\pi(v)=1$ for all $v$ and the result still holds. 
We can apply Lemma \ref{lem::exp_m} to data depth functions, which results in the following theorem. 
\begin{theorem}
Suppose that $\sup_v|\D(v;F_n)-\D(v;F)|=o_p(n^{-s})$ where $s\geq 0$, $\GS(\D)=C(\D)/n$, the maximum of $D(x;F)$ occurs uniquely at $\theta$ and $\D$ is $\alpha$-H\"{o}lder continuous at $\theta$, for some $\alpha>0$. 
Additionally, suppose that $\pi(v)$ is a bounded Lebesgue density which is positive in a neighborhood around $\theta$. 
Let $\rho\coloneqq\D(\theta;F)<\infty$, then for $\widetilde{T}(\sam{X}{n})$ drawn from the density
$$f(v;F_n)=\frac{\exp\left(-\dfrac{\epsilon_n}{2\GS(\D)}(\rho-  \D(v;F))\right)\pi(v)}{\int_{\mathbb{R}^d} \exp\left(-\dfrac{\epsilon_n}{2\GS(\D)}(\rho-\D(v;F_n))\right)\pi(v)dv},$$
it holds that
\begin{enumerate}
    \item $\widetilde{T}(\sam{X}{n})\conp \theta \text{ when } \epsilon_n=O(n^{r'})$, $r'<s-1$.
    \item $ \widetilde{T}(\sam{X}{n})\cond T,$  when $n\epsilon_n\rightarrow K<\infty$, where $T$ is a random vector whose probability density function is proportional to $\exp\left(-\frac{K(\rho-  \D(v;F))}{2C(\D)}\right).$
\end{enumerate}

\label{thm::exp_med}
\end{theorem}
\begin{remark}
The continuity condition is weak, in the sense that $\alpha$ can be very small. 
Halfspace depth is $\alpha$-H\"{o}lder continuous if $F_u$ are $\alpha$-H\"{o}lder continuous and $F$ is continuous. 
$\IRW$ depth is $\alpha$-H\"{o}lder continuous if $F_u$ are $\alpha$-H\"{o}lder continuous.
For simplicial depth, we only need $F$ to be $\alpha$-H\"{o}lder continuous. 
\label{rem::continuity}
\end{remark}
\begin{remark}
For many depth functions, we can choose $s$ arbitrarily close to 1/2 and the convergence requirement is still satisfied. 
Therefore choices of $r'$ such that $r'$ is close to -1/2 give the fastest rates at which the privacy parameter can decrease to 0 while maintaining consistency of the estimator. 
\end{remark}
Theorem \ref{thm::exp_med} can then be applied to the three depths of \citep{Tukey1974,Liu1988,RAMSAY201951}. 
These three depth functions all satisfy the uniform consistency requirement for $s<1/2$ \citep{DUMBGEN1992119, masse2004, RAMSAY201951} and continuity was discussed in Remark \ref{rem::continuity}. 
The assumption on uniqueness of the median is trickier, in the sense that these depth-based medians are not necessarily unique. 
We only need the population depth-based median to be unique, which is satisfied for distributions which are symmetric about a unique point. 
This holds because these depth functions satisfy the maximality at center property \citep[see, e.g.,][]{Zuo2000}. 
Algorithms that implement Mechanism \ref{mech::exp_med} are an interesting line of new research; we cannot directly use, say Markov Chain Monte Carlo methods, without first ensuring that they maintain the privacy of the estimators. 

For projection depth, we cannot use the exponential mechanism without injecting a significant level of noise into the estimator and so we instead extend the propose-test-release framework \citep{Brunel2020} to be used with the exponential mechanism. 
Suppose $\phi_{\sam{X}{n}}\colon \rdd\rightarrow\re^+$ is some cost function which we would like to minimize. 
Then, define 
$$A_\eta(\phi_{\sam{X}{n}};\sam{X}{n})=\min \left\{k\in \mathbb{N}\colon\  \sup_{\sam{Y}{n}\in\mathcal{D}(\sam{X}{n};k)}\sup_x |\phi_{\sam{X}{n}}(x)-\phi_{\sam{Y}{n}}(x)|>\eta\right\},$$
as the truncated breakdown point of the cost function. 
This is a direct extension of the truncated breakdown point of \citep{Brunel2020} to the functional context; the norm in \eqref{eqn::tbdp} is replaced with a norm on the function space. 
The following mechanism extends PTR to be used with the exponential mechanism:
\begin{mech}
Suppose that $$ Q_{\sam{X}{n}}(A)=\int_A\frac{\exp(-{\phi_{\sam{X}{n}}(v)\frac{\epsilon }{ 2\eta}})dv}{\int_{\rdd}\exp(-{\phi_{\sam{X}{n}}(v)\frac{\epsilon }{ 2\eta}})dv}$$
is a valid measure on $(\rdd,\mathscr{B}(\rdd)).$ Let $a_\delta,\ b_\delta$ and $V$ be as in Mechanism \ref{mech::PTR}. 
Then the estimator 
\begin{equation*}
 \widetilde{T}(\sam{X}{n})=\left\{  \begin{array}{ll}
   \perp  & \text{if } A_\eta(\phi_{\sam{X}{n}};\sam{X}{n})+\frac{a_\delta}{\epsilon} V\leq 1+\frac{b_\delta}{\epsilon} \\
   \widehat{T}(\sam{X}{n})\sim Q_{\sam{X}{n}}& o.w.
\end{array} \right. ,
\end{equation*}
is a differentially private estimate of $\argmin_v \phi_{\sam{X}{n}}(v)$. Under the Laplace version, the estimator is $(2\epsilon,\delta)$-differentially private and under the Gaussian version, the estimator is $(2\epsilon,2\delta)$-differentially private.
\label{mech::fun_ptr}
\end{mech}
\noindent This mechanism shows that we can still use propose-test-release when the cost function is likely to have low local sensitivity. 
In fact, the the Gaussian version uses slightly less of the privacy budget than that of the original propose-test-release mechanism, which is due to the pure differential privacy of the exponential mechanism. 
The following theorem can be used to prove consistency of an estimator generated via Mechanism \ref{mech::fun_ptr}. 
\begin{theorem}
Suppose that the sequence of random functions $\phi_{\sam{X}{n}}(v)\colon \mathbb{R}^d\rightarrow\mathbb{R}^+$ satisfy the conditions of Lemma \ref{lem::exp_m} and that 
$$\int_{\rdd}\exp\left(-\frac{K  \phi(v)}{2}\right)dv<\infty,$$
for any $K>0$. 
Suppose that $\widetilde{T}(\sam{X}{n})$ is generated according to Mechanism \ref{mech::fun_ptr}. 
If the sequences $\epsilon_n,\delta_n, \eta_n$ imply that $$\Pr\left(A_{\eta_n}(\phi_{\sam{X}{n}};\sam{X}{n})+V \frac{a_{\delta_n}}{\epsilon_n}\leq \frac{b_{\delta_n}}{\epsilon_n}+1\right)\rightarrow 0,$$ as $n\rightarrow\infty$ then it holds that
\begin{enumerate}
    \item $\widetilde{T}(\sam{X}{n})\conp \argmin_{v\in\rdd}\phi(v)\text{ when } \epsilon_n/\eta_n=O(n^{r}),\ r<1/2$.
    \item $\widetilde{T}(\sam{X}{n})\cond T,$  when $\epsilon_n/\eta_n \rightarrow K<\infty$, where $T$ is a random vector whose probability density function is proportional to $\exp\left(-\frac{K  \phi(v)}{2}\right).$
\end{enumerate}
\label{thm::PTR_EM}
\end{theorem}

We can now substitute in the outlyingness function and see how this algorithm works for the purposes of  privately estimating the projection depth median. 
A first question is whether or not the following probability density function
\begin{align*}
f(v)=\cfrac{\exp\left(-\dfrac{\epsilon  \sup_u|v^\top u-\med(\sam{X}{n}^\top u)|/\varsigma(\sam{X}{n}^\top u)}{2\eta}\right)}{\int_{\mathbb{R}^d}\exp\left(-\dfrac{\epsilon  \sup_u|v^\top u-\med(\sam{X}{n}^\top u)|/\varsigma(\sam{X}{n}^\top u)}{2\eta}\right)dv}
\end{align*}
even exists. 
Recall from \citep{zuo2003} that $$ \sup_u\frac{|v^\top u-\med(\sam{X}{n}^\top u)|}{\varsigma(\sam{X}{n}^\top u)}\geq \frac{\norm{v}-\sup_u\med(\sam{X}{n}^\top u)}{\sup_u \varsigma(\sam{X}{n}^\top u)}.$$
It follows that if $\sup_u \med(\sam{X}{n}^\top u)<\infty$, then
\begin{align*}
\int_{\mathbb{R}^d}\exp\left(-\frac{\epsilon  \sup_u|v^\top u-\med(\sam{X}{n}^\top u)|/\varsigma(\sam{X}{n}^\top u)}{2\eta}\right)dv &\leq C_1\int_{\mathbb{R}^d}\exp\left(-\frac{C_2  \norm{v}}{2}\right)dv<C_1\int_{\mathbb{R}^d}\exp\left(-\frac{C_3  \norm{v}_1}{2}\right)dv<\infty,
\end{align*}
where the second inequality follows from equivalency of norms and the last inequality follows from the fact that $\exp\left(-C_3  \norm{v}_1/2\right)$ is proportional to a Laplace density function. 
Unfortunately, immediately using PTR with the exponential mechanism gives no gains in estimating the projection median over using the global sensitivity of projection depth (which is 1). 
If the points in $\sam{X}{n}$ are distinct, we have that $A_{\eta}(O_\ell(\cdot;\sam{X}{n});\sam{X}{n})=1$ for any $\eta$. 
To see this, suppose that $\sam{Y}{n}$ is a neighboring dataset, with $X_1$ changed to be some observation such that $\varsigma(\sam{Y}{n}^\top u)\neq \varsigma(\sam{X}{n}^\top u)$. 
It follows that for any $u$
$$\sup_x \left|O^u_\ell(x;\sam{X}{n})-O^u_\ell(x;\sam{Y}{n})\right|\approx \sup_x \left|x^\top u\frac{\varsigma(\sam{X}{n}^\top u)-\varsigma(\sam{Y}{n}^\top u)}{\varsigma(\sam{X}{n}^\top u)\varsigma(\sam{Y}{n}^\top u)}\right|=\infty.$$
In order to estimate the projection depth-based median privately, we may truncate the outlyingness function $O$ in the following manner
$$\mathcal{O}_\ell(x;\sam{X}{n};M_n)=\left\{\begin{array}{cc}
   O_\ell(x;\sam{X}{n})  & \norm{x}<M_n \\
    \infty  & \norm{x}\geq M_n
\end{array}\right. .$$
We can now apply Mechanism \ref{mech::fun_ptr} and Theorem \ref{thm::PTR_EM} to $\mathcal{O}$ in order to privately estimate the projection depth-based median. 
The following theorem gives reasonable choices of $\eta$ and $\epsilon$ that maintain consistency of the estimator. 
\begin{theorem}
Suppose that $0< M_n=o(n^{1/2})$, $\sup_{u}|\med(F_u)|<\infty$, $\sup_{u}\iqr(F_u)<\infty$, $\inf_{u}\iqr(F_u)>0$ and the conditions of Theorem \ref{thm::pd_con} hold. 
Further, suppose that $\PD(x;F)$ is uniquely maximized at $\theta.$ 
Let $\widetilde{T}(\sam{X}{n})$ be the estimator of Mechanism \ref{mech::fun_ptr} with cost function $\phi_{\sam{X}{n}}(v)=\mathcal{O}_2(v;\sam{X}{n};M_n)$, then
\begin{enumerate}
    \item $\widetilde{T}(\sam{X}{n})\conp \theta$ when $\epsilon_n/\eta_n=O(n^{r}),\ r<1/2$. 
    \item $\widetilde{T}(\sam{X}{n})$ converges in distribution to a random vector whose distribution is defined by the measure
$$Q(A)=\int_A\frac{\exp(-(O_2(v;F)-O_2(\theta;F))\frac{K}{ 2})dv}{\int_{\rdd}\exp(-(O_2(v;F)-O_2(\theta;F)){\frac{K }{ 2}})dv}$$ when $\epsilon_n/\eta_n\rightarrow K,\ K<\infty$. 
\end{enumerate}
\label{thm::c_PDM}
\end{theorem}

The obvious issue is choosing $M_n$ in practice, which can be partially informed by the above theorem. 
Clearly, if the data is known to be bounded it is easy to choose $M_n$. 
If the data are not bounded one can choose $M_n$ independent of the data, given domain knowledge. 
It is important that the choice of $M_n$ does not depend on the data, which would violate the consistency theorem; if $M_n$ is chosen based on the data, then $M_n$ could differ between two datasets, implying that  $O^u_\ell(x;\sam{X}{n})-O^u_\ell(x;\sam{Y}{n})=\infty$ for some $x$ and consequentially the truncated breakdown point of the outlyingness function is 1. 
Computationally, again the difficulty lies in computing $A_\eta(\phi_{\sam{X}{n}};\sam{X}{n})$, for which we can use similar methods as discussed in the previous section. 
\section{Concluding Remarks}
We have introduced several mechanisms for differentially private, depth-based inference. 
These mechanisms include private estimates of point-wise depth values for population depth functions, such as halfspace and projection depth. 
Such mechanisms have been shown to output consistent estimators, even for cases where the privacy budget is small, i.e., $\epsilon_n\rightarrow 0$. 
Notable is that we have shown that one can get consistent estimates of projection depth values, even though it has high global sensitivity. 
We have also introduced algorithms for estimating popular depth-based medians, including the simplicial, halfspace, IRW and projection medians. 
These algorithms all provide differentially private, consistent estimators of the population median under some very mild conditions. 
Here, the privacy budget is also permitted to decrease to 0, provided it is not too fast, e.g., $\epsilon_n\propto n^{-r},\ r<1/2$ for the halfspace median. 
We further provide some general tools for constructing and studying differentially private estimators. 
We provide a lemma for showing weak consistency of differentially private estimators based on the exponential mechanism, even if the objective function is not differentiable. 
We also extend the propose-test-release algorithm of \cite{Brunel2020} to be used with the exponential mechanism, which allows one to privately estimate maximizers of objective functions which have infinite global sensitivity. 
We also provide tools to show weak consistency of an estimator based on our extension of the PTR algorithm. 
We apply the extended PTR algorithm and the related consistency result to the projection depth-based median. 

The mechanisms introduced in this paper can be used to perform different types of private inference via the depth-based inference framework. 
One benefit is that these inference procedures will retain their robustness; robustness may be even more useful in the private setting, where the analyst may have only limited access to the database and therefore cannot determine if the data contains outliers. 
Furthermore, this work has shown another meaningful connection between robust statistics and differential privacy. 
As such, it has opened up many avenues for further research, which the authors are exploring. 
Some of these include how these algorithms perform in different inferential contexts, such as depth-based hypothesis testing and clustering. 
Another area of interest is the computation of the medians presented in Section \ref{sec::pm}. 
Particularly the computation of the truncated breakdown point, which is more difficult than in the one-dimensional setting.

\bibliographystyle{apalike}
\bibliography{main}
\section{Proofs}
\begin{proof}[Proof of Theorem \ref{thm::gshs}]
The first property follows directly from consistency of the sample depths and the fact that $\GS(D)\rightarrow 0$. 
The second case is true for the same reasons.
\end{proof}
\begin{proof}[Proof of Theorem \ref{thm::pd_con}]
We begin with the Laplace case. 
In order to show that $\widetilde{\PD}_2(x;F_n)\conp \PD(x;F;\med,\iqr)$,
we must have that $\frac{\eta}{\epsilon_n}W_2\rightarrow 0$, $\PD(x;F_n;\med,\iqr)\conp \PD(x;F;\med,\iqr)$ and 
\begin{equation}
\label{eqn::step1}
    \Pr\left( A_{\eta}\left(O_{2}\left(x ; F_{n}\right) ; \sam{X}{n}\right)  \leq 1+\frac{\log(2/\delta_n)-W_1}{\epsilon_n}\right)\rightarrow 0,
\end{equation}
as $n\rightarrow\infty.$ 
The first two properties hold from the assumptions and the properties of projection depth. 
It remains to show \eqref{eqn::step1}. 
To this end, note that $$\Pr(|W_1|>\log(2/\delta_n))=2e^{-\log(2/\delta_n)}=\delta_n=O(n^{-k}),$$
from the properties of the Laplace distribution and the rate of convergence of $\delta_n$. 
We can then write
{\small
\begin{align}
    \Pr\left( A_{\eta}\left(O_{2}\left(x ; F_{n}\right) ; \sam{X}{n}\right)  \leq 1+\frac{\log(2/\delta_n)-W_1}{\epsilon_n}\right)&\leq \Pr\left( A_{\eta}\left(O_{2}\left(x ; F_{n}\right) ; \sam{X}{n}\right)  \leq 1+\frac{\log(2/\delta_n)-W_1}{\epsilon_n},W_1>-\log(2/\delta_n)\right)\nonumber\\
    &\qquad +\Pr\left( A_{\eta}\left(O_{2}\left(x ; F_{n}\right) ; \sam{X}{n}\right)  \leq 1+\frac{\log(2/\delta_n)-W_1}{\epsilon_n},W_1<-\log(2/\delta_n)\right)\nonumber\\
    \label{eqn::step11}
    &\leq \Pr\left( A_{\eta}\left(O_{2}\left(x ; F_{n}\right) ; \sam{X}{n}\right)  \leq 1+2\frac{\log(2/\delta_n)}{\epsilon_n}\right)+O(n^{-k}).
\end{align}}
Now, let $\rho_n=2\frac{\log(2/\delta_n)}{\epsilon_n}$ and we want to show that 
\begin{equation}
\label{eqn::step12}
\Pr\left(A_{\eta}(O_{2}\left(x ; F_{n}\right);\sam{X}{n}) \leq 1+\rho_n \right)\rightarrow 0,
\end{equation}
which, together with \eqref{eqn::step11} implies \eqref{eqn::step1}. 
To this end, it holds that
\begin{align}
\Pr\left(A_{\eta}(O_{2}\left(x ; F_{n}\right);\sam{X}{n}) \leq 1+\rho_n \right)&=\Pr\left(\bigcup_{j=1}^{1+\floor{\rho_n}}\left\{\sup_{\sam{Y}{n}\in \mathcal{D}(\sam{X}{n},j)}|O_{2}(x;\sam{X}{n})-O_{2}(x;\sam{Y}{n})|\geq \eta\right\} \right)\nonumber\\
\label{eqn::nested}
&=\Pr\left(\sup_{\sam{Y}{n}\in \mathcal{D}(\sam{X}{n},1+\floor{\rho_n})}|O_{2}(x;\sam{X}{n})-O_{2}(x;\sam{Y}{n})|\geq \eta \right),    
\end{align}
where the last line follows from the fact that $$\left\{\sup_{\sam{Y}{n}\in \mathcal{D}(\sam{X}{n},j+1)}|O_{2}(x;\sam{X}{n})-O_{2}(x;\sam{Y}{n})|\geq \eta\right\} \subset \left\{\sup_{\sam{Y}{n}\in \mathcal{D}(\sam{X}{n},j)}|O_{2}(x;\sam{X}{n})-O_{2}(x;\sam{Y}{n})|\geq \eta\right\},$$
for $j\in\{1,\ldots,1+\floor{\rho_n}\}$. 
We now must only show that
\begin{equation}
    \Pr\left(\sup_{\sam{Y}{n}\in \mathcal{D}(\sam{X}{n},1+\floor{\rho_n})}|O_{2}(x;\sam{X}{n})-O_{2}(x;\sam{Y}{n})|\geq \eta \right)\rightarrow 0,
    \label{eqn::step3}
\end{equation} 
as $n\rightarrow\infty$. 
Combining this with \eqref{eqn::nested} and \eqref{eqn::step12} implies that \eqref{eqn::step1} holds from the previous argument. 
Consider the Taylor series expansion of $f(x,y)=x/y$ about the point $(x_0,y_0)=(|x^\top u-\med\left(F_u\right)|,\iqr(F_u))$:
\begin{align*}
O^u_2(x;\sam{X}{n})&=\frac{|x^\top u-\med\left(F_u\right)|}{\iqr(F_u)}+\frac{|x^\top u-\med\left(\sam{X}{n}^\top u\right)|-|x^\top u-\med\left(F_u\right)|}{\iqr(F_u)}\\
&\qquad\qquad\qquad\qquad\qquad\qquad -\frac{|x^\top u-\med\left(F_u\right)|}{\iqr(F_u)^2}(\iqr(\sam{X}{n}^\top u)-\iqr(F_u))+\mathcal{R}_{n,u}\\
&=\frac{|x^\top u-\med\left(\sam{X}{n}^\top u\right)|}{\iqr(F_u)}-\frac{|x^\top u-\med\left(F_u\right)|}{\iqr(F_u)^2}(\iqr(\sam{X}{n}^\top u)-\iqr(F_u))+O_p(n^{-1}).
\end{align*}
It is easy to see that $\mathcal{R}_{n,u}=O_p(n^{-1})$, since $(\iqr(\sam{X}{n}^\top u)-\iqr(F_u))^2=O_p(n^{-1})$, which holds for our assumptions on the quantiles of $F_u$. 
We can then write 
\begin{align*}
|O^u_2(x;\sam{X}{n})-O^u_2(x;\sam{Y}{n})|&=\left|\frac{|x^\top u-\med\left(u^\top \sam{X}{n}\right)|}{\iqr(\sam{X}{n}^\top u)}-\frac{|x^\top u-\med\left(u^\top \sam{Y}{n}\right)|}{\iqr(\sam{Y}{n}^\top u)}\right|\\
&=\Bigg|\frac{|x^\top u-\med\left(\sam{X}{n}^\top u\right)|-|x^\top u-\med\left(\sam{Y}{n}^\top u\right)|}{\iqr(F_u)}\\
&\qquad\qquad\qquad\qquad\qquad\qquad-\frac{|x^\top u -\med\left(F_u\right)|}{\iqr(F_u)^2}(\iqr(\sam{X}{n}^\top u)-\iqr(\sam{Y}{n}^\top u))+O_p(n^{-1})\Bigg|\\
&\leq \Bigg|\frac{|x^\top u-\med\left(\sam{X}{n}^\top u\right)|-|x^\top u-\med\left(\sam{Y}{n}^\top u\right)|}{\iqr(F_u)}\Bigg|\\
&\qquad\qquad\qquad\qquad\qquad\qquad+\Bigg|\frac{|x^\top u -\med\left(F_u\right)|}{\iqr(F_u)^2}(\iqr(\sam{X}{n}^\top u)-\iqr(\sam{Y}{n}^\top u))\Bigg|+O_p(n^{-1})\\
&\leq \frac{|\med\left(\sam{Y}{n}^\top u\right)-\med\left(\sam{X}{n}^\top u\right)|}{\iqr(F_u)}+\Bigg|\frac{|x^\top u -\med\left(F_u\right)|}{\iqr(F_u)^2}(\iqr(\sam{X}{n}^\top u)-\iqr(\sam{Y}{n}^\top u))\Bigg|+O_p(n^{-1}),
\end{align*}
where the last line follows from the reverse triangle inequality. 
Now, recall that $\sam{Y}{n}$ differs from $\sam{X}{n}$ by at most $1+\floor{\rho_n}$ points. 
If $F_{n,u}$ is the empirical distribution corresponding to $\sam{X}{n}^\top u$ and $G_{n,u}(x)=1-F_{n,u}(x)$ , it holds that
\begin{align*}
|\med\left(\sam{Y}{n}^\top u\right)-\med\left(\sam{X}{n}^\top u\right)|&\leq \max\left\{|F_{n,u}^{-1}(1/2)-F_{n,u}^{-1}(1/2+(\rho_n+1)/n)|,|F_{n,u}^{-1}(1/2)-F_{n,u}^{-1}(1/2-(\rho_n+1)/n)|\right\}\\
&\overset{def.}{=}|F_{n,u}^{-1}(1/2)- F_{n,u}^{-1}(1/2\pm(\rho_n+1)/n)| \\
&=|F_{u}^{-1}(1/2)-F_{n,u}^{-1}(1/2) + F_{n,u}^{-1}(1/2\pm(\rho_n+1)/n)-F_{u}^{-1}(1/2)| \\
&=|G_{n,u}(\med(F_u))+\mathcal{R}''_{n,u}-1/2-G_{n,u}(\med(F_u))-\mathcal{R}'_{n,u}+1/2| \\
&=O(n^{-3/4}\log n)\ a.s.\ .
\end{align*}
where the second last line and the last line follow from a Bahadur type representation of quantiles, as long as $\frac{1}{n}\left(1+2\frac{\log(2/\delta_n)}{\epsilon_n}\right)=O((\log\log n /n)^{3/4})$ \citep[see Theorem 2 on page 2 of][]{DeHaan1979}. 
We know that $\frac{1}{n}\left(1+2\frac{\log(2/\delta_n)}{\epsilon_n}\right)=O((\log\log n /n)^{3/4})$ holds from the assumptions on $\epsilon_n$ and $\delta_n$. 
We can show something similar for the inter-quartile range by simply replacing 1/2 with 1/4 and 3/4. 
Now, we must show that 
\begin{align*}
    \left|\sup_u O^u_2(x;\sam{X}{n})- \sup_u O^u_2(x;\sam{Y}{n})\right|\rightarrow 0\ a.s.\ .
\end{align*}
We see that
\begin{align*}
    \left|\sup_u O^u_2(x;\sam{X}{n})- \sup_u O^u_2(x;\sam{Y}{n})\right|
    &\leq 2\sup_u \left|O^u_2(x;\sam{X}{n})- O_2^u(x;\sam{Y}{n})\right|=O(n^{-3/4}\log n)\ a.s.\ ,
\end{align*}
where the last line follows from the fact that $\xi_{1/4,u},\ \xi_{3/4,u}$ are bounded as functions of $u$, implying that $\mathcal{R}'_{n,u},\mathcal{R}'_{n,u}$ are also bounded in $u$. 
See the proof of Theorem 2' of \citep{DeHaan1979} for the exact expression of $\mathcal{R}'_{n,u},\mathcal{R}''_{n,u}$. 
This implies that for $\eta \propto\frac{\log n}{n^{3/4-r}}$, \eqref{eqn::step3} holds, which immediately gives that  $\widetilde{\PD}_2(x;F_n)\conp \PD(x;F;\med,\iqr)$ as $n\rightarrow\infty$ from the argument preceding \eqref{eqn::step3}. 

For the Gaussian case, we have that  $$\Pr(|Z_1|\sqrt{2\log(1.25/\delta_n)}|>2\log(1.25/\delta_n))\leq2e^{-\log(1.25/\delta_n)}=2(\delta_n/1.25)=O(n^{-k}),$$
from the properties of the normal distribution and the rate of convergence of $\delta_n$. 
We can then write, using the same argument as above, that
{\footnotesize
\begin{align*}
    \Pr\left( A_{\eta}\left(O_{2}\left(x ; F_{n}\right) ; \sam{X}{n}\right)  \leq 1+\frac{2\log(1.25/\delta_n)-Z_1\sqrt{2\log(1.25/\delta_n)}}{\epsilon_n}\right)
    &\leq \Pr\left( A_{\eta}\left(O_{2}\left(x ; F_{n}\right) ; \sam{X}{n}\right)  \leq 1+4\frac{\log(1.25/\delta_n)}{\epsilon_n}\right)\\
    &\qquad\qquad\qquad\qquad+O(n^{-k}).
\end{align*}}
This expression is of the same form as that of \eqref{eqn::step11} in the Laplace case, and the same arguments apply. 
\end{proof}
\begin{proof}[Proof of Lemma \ref{lem::exp_m}]
It is clear that $$\dfrac{e^{-\lambda_n\phi_n(\omega,v)}\pi(v)dv}{\int_{\mathbb{R}^d}  e^{-\lambda_n\phi_n(\omega,v)}\pi(v)dv}$$ is a valid density, since $e^{-x}$ is bounded for all $x\in \re^+$. 
The goal is to show that $Q_n$ converges weakly to $\indtw_\mathbf{0}(\cdot)$, since this is equivalent to $V_n\conp\mathbf{0}$. 
Note that we use $\indtw_\mathbf{0}(B)$ as shorthand for $\ind{\mathbf{0}\in B}.$ 
We use the Portmanteau Theorem and show that for all $\indtw_\mathbf{0}(\cdot)$-continuity sets $A$, $Q_n(A)\rightarrow \indtw_\mathbf{0}(A)$. Let $B_n=\{\omega\colon \norm{\phi_n(\omega,\cdot)-\phi(\omega,\cdot)}<n^{-s}\}$ and let $Q_n^\omega(\cdot)$ be the measure corresponding to $V_n$, conditional on $\omega$, viz.
$$Q_n^\omega(A)=\int_A\dfrac{e^{-\lambda_n\phi_n(\omega,v)}\pi(v)dv}{\int_{\mathbb{R}^d}  e^{-\lambda_n\phi_n(\omega,v)}\pi(v)dv}.$$
We can now write
\begin{align*}
    \limn Q_n(A)&= \limn\int_\Omega\int_A\dfrac{e^{-\lambda_n\phi_n(\omega,v)}\pi(v)dv}{\int_{\mathbb{R}^d}  e^{-\lambda_n\phi_n(\omega,v)}\pi(v)dv}dP= \limn\int_{B_n}Q_n^\omega(A)dP+\limn\int_{B^c_n}Q_n^\omega(A)dP.
\end{align*}
It is easy to see that 
$$ 0\leq \limn\int_{B^c_n}Q_n^\omega(A)dP\leq \limn\int_{B^c_n}dP=0,$$
where the last equality comes from assumption 2. 
Using this, we can write 
\begin{align*}
    \limn Q_n(A)&= \limn\int_{B_n}Q_n^\omega(A)dP+\limn\int_{B^c_n}Q_n^\omega(A)dP= \limn\int_{B_n} Q_n^\omega(A)dP=\int_{\Omega} \limn \ind{\omega\in B_n}Q_n^\omega(A)dP,
\end{align*}
where the last equality follows from dominated convergence theorem, noting that $Q_n^\omega(A)<1$. 
We now consider $\limn Q_n^\omega(A)$ for fixed $\omega\in B_n$. 
Note that for any $\indtw_\mathbf{0}(\cdot)$-continuity set, $\mathbf{0}$ is either an interior point or not in the set. 
Keeping this in mind, let $A_I$ be a $\indtw_\mathbf{0}(\cdot)$-continuity set such that $\mathbf{0}$ is interior in $A_I$. 
We can write
\begin{align*}
\limn Q_n^\omega(A_I)&=\limn\int_{A_I}\dfrac{e^{-\lambda_n\phi_n(\omega,v)}\pi(v)dv}{\int_{\mathbb{R}^d}  e^{-\lambda_n\phi_n(\omega,v)}\pi(v)dv}=\limn\int_{A_I}\dfrac{e^{-\lambda_n\phi_n(\omega,v)}\pi(v)dv}{\int_{A_I}  e^{-\lambda_n\phi_n(\omega,v)}\pi(v)dv+\int_{A_I^c}  e^{-\lambda_n\phi_n(\omega,v)}\pi(v)dv}. 
\end{align*}

Observe that for fixed $\omega\in B_n$, $\lambda_n(\phi_n(\omega,v)-\phi(\omega,v))=O(n^r)o(n^{-s})=o(n^{r-s})\coloneqq o(n^{-\beta})$, where $\beta>0$; we then have that $\lambda_n(\phi_n(\omega,v)-\phi(\omega,v))=o(n^{-\beta})$. 
Thus, assumption 1 and 2 imply that $\lambda_n(\phi_n(\omega,v)-\phi(\omega,v))\leq Cn^{-\beta}$ for some $\beta>0$ independent of $v$.  
Therefore, we can write
\begin{align*}
     \int_{A_I}  e^{-\lambda_n\phi_n(\omega,v)}\pi(v)dv&= \int_{A_I}  e^{-\lambda_n(\phi_n(\omega,v)-\phi(\omega,v)+\phi(\omega,v)-\phi(\omega,\mathbf{0}))}\pi(v)dv\\
    &= \int_{A_I}  e^{-\lambda_n(\phi_n(\omega,v)-\phi(\omega,v))}e^{-\lambda_n(\phi(\omega,v)-\phi(\omega,\mathbf{0}))}\pi(v)dv\\
    &\geq  \int_{A_I}  e^{-Cn^{-\beta}}e^{-\lambda_n(\phi(\omega,v)-\phi(\omega,\mathbf{0}))}\pi(v)dv.
\end{align*}
For large $n$ and fixed $\xi>0$, the neighborhood $\mathcal{N}_{n^{-\xi}}(\mathbf{0})=\{x\in \rdd\colon \norm{x}<n^{-\xi}\}$ is in $A_I$, since $\mathbf{0}$ is interior in $A_I$. 
From assumption 3 (H\"{o}lder continuity) we have that 
$$\sup_{v\in\mathcal{N}_{n^{-\xi}}(\mathbf{0})} |\phi(\omega,v)-\phi_n(\omega,\mathbf{0})|<C_2 d^{\xi\alpha}/n^{\xi\alpha}.$$
Choose $\xi$ such that $\alpha'=\xi\alpha>r$, and write
\begin{align*}
     \int_{A_I} e^{-\lambda_n(\phi(\omega,v)-\phi(\omega,\mathbf{0}))}e^{-Cn^{-\beta}}\pi(v)dv&\geq  \int_{\mathcal{N}_{n^{-\xi}}(\mathbf{0})}  e^{-\lambda_n(\phi(\omega,v)-\phi(\omega,\mathbf{0}))}e^{-Cn^{-\beta}}\pi(v)dv\\
   &\geq  C_3n^{-d\xi} e^{-C_2 n^{-\alpha'}\lambda_n}e^{-Cn^{-\beta}}\\
   &\geq   C_3n^{-d\xi} e^{-C_2 n^{-\alpha'+r}}e^{-Cn^{-\beta}}\\
   &=O(n^{-d\xi}).
\end{align*}
Note that assumption 5 implies that there exists some $N$, such that for all $n>N$, $\pi(v)$ is bounded below on $\mathcal{N}_{n^{-\xi}}(\mathbf{0})$ and so $\pi$ is absorbed into the constant $C_3$. 
Now, consider $A_{I}^c$, a $\indtw_\mathbf{0}(\cdot)$-continuity set such that $\mathbf{0}$ is not interior in $A_{I}^c$. 
There then exists a neighborhood around $\mathbf{0}$, call it $\mathcal{N}_k(\mathbf{0})$, such that $\mathcal{N}_k(\mathbf{0})\notin A_{I}^c$. 
By assumption 3 and 4, we also have that $\phi(\omega,v)>k'>0$ on $A_{I}^c$, for some $k'$ independent of $n$. 
It follows easily that
\begin{align*}
    \int_{A_{I}^c}  e^{-\lambda_n\phi_n(\omega,v)}\pi(v)dv&=  \int_{A_{I}^c}  e^{-\lambda_n\phi(\omega,v)}e^{-\lambda_n(\phi_n(\omega,v)-\phi(\omega,v))}\pi(v)dv\\
    &= O(1)\int_{A_{I}^c}  e^{-\lambda_n\phi(\omega,v)}\pi(v)dv\\
    &\leq  O(1)\int_{A_{I}^c}  e^{-\lambda_nk'}\pi(v)dv \\
    &= O(e^{-\lambda_nk'}),
\end{align*} 
where the second equality comes from the fact that $\lambda_n(\phi_n(\omega,v)-\phi(\omega,v))=o(1)$ uniformly in $v$. 
The last equality comes from the fact that $\pi(v)$ is a density with respect to the Lebesgue measure. 
It then follows that 
\begin{align*}
    \limn Q_n^\omega(A_{I})&=\limn\int_{A_{I}}\dfrac{e^{-\lambda_n\phi_n(\omega,v)}\pi(v)dv}{\int_{A_{I}}  e^{-\lambda_n\phi_n(\omega,v)}\pi(v)dv+\int_{A_{I}^c}  e^{-\lambda_n\phi_n(\omega,v)}\pi(v)dv}=\limn\frac{O(n^{-d\xi})}{O(n^{-d\xi})+O(e^{-\lambda_n k'}) }=1,
\end{align*}
which immediately gives that $\limn Q_n^\omega(A)=\indtw_\mathbf{0}(A)$ for $\indtw_\mathbf{0}(\cdot)$-continuity sets $A$. 
Then, since $\limn \ind{\omega\in B_n}=1$ $P$-almost surely by assumption 2, we have that
$$\limn Q_n(A)=\int_{\Omega} \limn \ind{\omega\in B_n}Q_n^\omega(A)dP=\indtw_\mathbf{0}(A),$$
which implies that 
$$Q_n \cond \indtw_\mathbf{0}(\cdot).$$
\end{proof}

\begin{lem}
Suppose the conditions of Lemma \ref{lem::exp_m} hold, except that $\limn \lambda_n=K$. 
Let $V_n$ be a sequence of random variables whose measure on $(\mathbb{R}^d,\mathscr{B}(\mathbb{R}^d))$ is given by
$$Q_n(A)= \int_\Omega\int_A\dfrac{e^{-\lambda_n\phi_n(\omega,v)}\pi(v)dv}{\int_{\mathbb{R}^d}  e^{-\lambda_n\phi_n(\omega,v)}\pi(v)dv}dP,$$
for $A\in \mathscr{B}(\mathbb{R}^d)$. 
Then $V_n\cond Q(A)$, where 
$$Q(A)= \int_A\dfrac{e^{-K\phi(v)}\pi(v)dv}{\int_{\mathbb{R}^d}  e^{-K\phi(v)}\pi(v)dv}.$$
\label{lem::exp_m2}
\end{lem}
\begin{proof}
From the proof of Lemma \ref{lem::exp_m}, it suffices to look at 
$$\limn \int_A e^{-\lambda_n\phi_n(\omega,v)}\pi(v)dv, $$
on $B_n$. 
On $B_n$, $\norm{\phi_n(\omega,v)-\phi(\omega,v)}_\infty\rightarrow 0$, which implies that $a_n=e^{-\lambda_n(\phi_n(\omega,v)-\phi(\omega,v))}\rightarrow 1$, uniformly in $v$. 
We can then say that 
\begin{align*}
    \limn \int_A e^{-\lambda_n\phi_n(\omega,v)}\pi(v)dv&=\limn  \int_A  e^{-\lambda_n\phi(\omega,v)}a_n\pi(v)dv= \int_A  \limn e^{-\lambda_n\phi(\omega,v)} \pi(v)a_nd v =\int_A  e^{-K\phi(\omega,v)} \pi(v)dv,
\end{align*}
where the last line follows from dominated convergence theorem and the fact that $\pi$ is a density function. 
The result follows from the proof of Lemma \ref{lem::exp_m}.
\end{proof}
\begin{proof}[Proof of Theorem \ref{thm::exp_med}]
We can write this problem as an application of Lemmas \ref{lem::exp_m} and \ref{lem::exp_m2}. 
For the first part, set $\phi_n=\rho-\D(v;F_n)$. 
Then, we have that $\frac{1}{2C(\D)}n\epsilon_n=n^{r'+1}\frac{1}{2C(\D)}$ with $r'+1<s$. 
Assumption 1 of Lemma \ref{lem::exp_m} is then satisfied and the result follows. 
The second result is a direct application of Lemma \ref{lem::exp_m2}. 
\end{proof}
\begin{proof}[Proof of Remark \ref{rem::continuity}]
Note that for halfspace depth consider two points $x,y\in \rdd$. If $F_u$ are $\alpha$-H\"{o}lder continuous, then
\begin{align}
\label{eqn::fuc}
    |F_u(x^\top u)-F_u(y^\top u)|&\leq C| (x-y)^\top u|^\alpha = C\norm{x-y}^\alpha \left| \left(\frac{x-y}{\norm{x-y}}\right)^\top u\right|^\alpha\leq C\norm{x-y}^\alpha.
\end{align}
Now, without loss of generality, suppose that $\HD(x,F)>\HD(y,F)$. 
Suppose further that $u^*$ is such that $F_{u^*}(y^\top u)=\inf_u F_u(y^\top u)$. 
There exists such a $u^*$ because $F$ is continuous, implying that $F_u$ is continuous in $u$, thus, $F_u$ is continuous function on a compact set. 
It follows that 
\begin{align*}
|\HD(x,F)-\HD(y,F)|=\inf_u F_u(x^\top u)-F_{u^*}(y^\top u)\leq  F_{u^*}(x^\top u)-F_{u^*}(y^\top u)\leq C\norm{x-y}^\alpha.
\end{align*}
For $\IRW$ depth, it holds that
\begin{align*}
|\IRW(x,F)-\IRW(y,F)|&=\int_{S^{d-1}} \min(F_u(x^\top u),1-F_u(x^\top u))-\min(F_u(y^\top u),1-F_u(y^\top u))d\nu(u)\\
&\leq \int_{S^{d-1}} 2|F_u(x^\top u)-F_u(y^\top u)|+2|1-F_u(x^\top u)-F_u(y^\top u)|d\nu(u)\\
&\leq  4C\norm{x-y}^\alpha \int_{S^{d-1}}2d\nu(u),
\end{align*}
which is a result of \eqref{eqn::fuc} and the fact that $|1-F_u(x^\top u)-F_u(y^\top u)|\leq 1$. 

For simplicial depth, if $F$ is $\alpha$-H\"{o}lder continuous, then we must show that $\Pr(x\in \Delta(X_1,\ldots,X_{d+1})$ is also $\alpha$-H\"{o}lder continuous. 
It is easy to begin with two dimensions. 
Consider $\Pr(x\in \Delta(X_1,X_2,X_{3}))-\Pr(y\in \Delta(X_1,X_2,X_{3}))$, as per \citep{liu1990}, we need to show that $\Pr(\overline{X_1X_2}\text{ intersects }\overline{xy})\leq C\norm{x-y}^\alpha$. 
In order for this event to occur, we must have that $X_1$ is above $\overline{xy}$ and $X_2$ is below $\overline{xy}$, but both are projected onto the line segment $\overline{xy}$ when projected onto the line running through $\overline{xy}$. 
The affine invariance of simplicial depth implies we can assume, without loss of generality, that $x$ and $y$ lie on the axis of the first coordinate. 
Let $x_1$ and $y_1$ be the first coordinates of $x$ and $y$. Suppose that $X_{11}$ is the first coordinate of $X_1$. 
It then follows from $\alpha$-H\"{o}lder continuity of $F$ that 
$$\Pr(\overline{X_1X_2}\text{ intersects }\overline{xy})\leq \Pr(x_1<X_{11}<y_1)\leq C|x_1-y_1|^{\alpha}\leq C\norm{x-y}^\alpha.$$
In dimensions greater than two, a similar line of reasoning can be used. 
We can again assume, without loss of generality, that $x$ and $y$ lie on the axis of the first coordinate. 
It holds that $$\Pr(x\in \Delta(X_1,X_2,X_{3}))-\Pr(y\in \Delta(X_1,X_2,X_{3}))\leq \binom{d+1}{d}\Pr(A_d),$$ 
where $A_d$ is the event that the $d-1$-dimensional face of the random simplex, formed by $d$ points randomly drawn from $F$, intersects the line segment $\overline{xy}$. 
It is easy to see that 
\begin{align*}
    \Pr(A_d)\leq \Pr(x_1<X_{11}<y_1)\leq C|x_1-y_1|^{\alpha}\leq C\norm{x-y}^\alpha. \tag*{\qedhere}
\end{align*}
\end{proof}
\begin{proof}[Proof of Differential Privacy of Mechanism \ref{mech::fun_ptr}]
The proof has the same outline as that of \citep{Brunel2020}, as well as the proof that the exponential mechanism is differentially private, which can be found in \citep{McSherry2007, Dwork2014}. 
First, assume that it holds $|\phi_{\sam{X}{n}}(x)-\phi_{\sam{Y}{n}}(x)|\leq \eta$ $\forall x$, then
\begin{align*}
f_{\sam{X}{n}}(v)/f_{\sam{Y}{n}}(v)&=\frac{ \exp(-{\phi_{\sam{X}{n}}(v)\frac{\epsilon/2}{ \eta}})}{ \exp(-{\phi_{\sam{Y}{n}}(v)\frac{\epsilon/2}{ \eta}})}\frac{ \int \exp(-{\phi_{\sam{Y}{n}}(v)\frac{\epsilon/2}{ \eta}})dv}{ \int\exp(-{\phi_{\sam{X}{n}}(v)\frac{\epsilon/2}{ \eta}})dv}\\
&\leq e^{\epsilon/2}\frac{ \int \exp(-{\phi_{\sam{Y}{n}}(v)\frac{\epsilon/2}{ \eta}})dv}{ \int\exp(-{\phi_{\sam{X}{n}}(v)\frac{\epsilon/2}{ \eta}})dv}\\
&\leq  e^{\epsilon/2} e^{\epsilon/2}\frac{ \int \exp(-{\phi_{\sam{X}{n}}(v)\frac{\epsilon/2}{ \eta}})dv}{ \int\exp(-{\phi_{\sam{X}{n}}(v)\frac{\epsilon/2}{ \eta}})dv}\\
&=e^{\epsilon}.
\end{align*}
Note that, for $B\in \mathscr{B}(\rdd)$ (the Borel sets with respect to $\rdd$) this implies that 
\begin{equation}
    \Pr(\widehat{T}(\sam{X}{n})\in B)\leq e^{\epsilon}\Pr(\widehat{T}(\sam{Y}{n})\in B).
    \label{eqn::dpexp}
\end{equation}
It follows from \cite{Brunel2020} that $A_\eta(\phi_{\sam{X}{n}};\sam{X}{n})$ has global sensitivity equal to 1, since changing one point can at most change the breakdown by 1. 
Then
\begin{align*}
    \Pr\left( \widetilde{T}(\sam{X}{n})\in B\right)&=\Pr\left(A_\eta(\phi_{\sam{X}{n}};\sam{X}{n})+\frac{1}{\epsilon} V\geq 1+\frac{\log(2/\delta)}{\epsilon},\widehat{T}(\sam{X}{n})\in B\right)\\
    &\leq e^{\epsilon}\Pr\left(A_\eta(\phi_{\sam{Y}{n}};\sam{Y}{n})+\frac{1}{\epsilon} V\geq 1+\frac{\log(2/\delta)}{\epsilon}\right)\Pr(\widehat{T}(\sam{X}{n})\in B)\\
    &\leq e^{2\epsilon}\Pr\left(A_\eta(\phi_{\sam{Y}{n}};\sam{Y}{n})+\frac{1}{\epsilon} V\geq 1+\frac{\log(2/\delta)}{\epsilon}\right) \Pr(\widehat{T}(\sam{Y}{n})\in B)\\
    &=e^{2\epsilon}\Pr\left( \widetilde{T}(\sam{Y}{n})\in B\right).
\end{align*}
The first inequality is from independence and the fact that $A_\eta(\phi_{\sam{X}{n}};\sam{X}{n})+\frac{1}{\epsilon} V$ is an $\epsilon$-differentially private estimator. 
The second inequality is from \eqref{eqn::dpexp}.
Now what if there exists an $x$ such that $|\phi_{\sam{X}{n}}(x)-\phi_{\sam{Y}{n}}(x)|\geq \eta$ ? This implies that $A_\eta(\phi_{\sam{X}{n}};\sam{X}{n})=1$ and
\begin{align*}
    \Pr\left( \widetilde{T}(\sam{X}{n})\in B\right)&\leq \Pr\left(A_\eta(\phi_{\sam{X}{n}};\sam{X}{n})+\frac{1}{\epsilon} V\geq 1+\frac{\log(2/\delta)}{\epsilon}\right)=\Pr\left( V\geq \log(2/\delta)\right) =\delta  \leq \delta +e^{2\epsilon}\Pr\left( \widetilde{T}(\sam{Y}{n})\in B\right).
\end{align*}
This implies that we get $(2\epsilon,\delta)$ differential privacy if $B$ is restricted to $\mathscr{B}(\rdd)$. 
For completeness, we need to include sets of the form $B=B'\cup\{\perp\}$, where $B'\in \mathscr{B}(\rdd)$. 
Consider 
\begin{align*}
\Pr\left( \widetilde{T}(\sam{X}{n})\in B\right)&=\Pr\left( \widehat{T}(\sam{X}{n})\in B',A_\eta(\phi_{\sam{X}{n}};\sam{X}{n})+\frac{1}{\epsilon} V\leq 1+\frac{\log(2/\delta)}{\epsilon}\right)+ \Pr\left(A_\eta(\phi_{\sam{X}{n}};\sam{X}{n})+\frac{1}{\epsilon} V> 1+\frac{\log(2/\delta)}{\epsilon}\right) \\
&\leq e^{2\epsilon}\left(\Pr\left( \widetilde{T}(\sam{Y}{n})\in B'\right)+ \Pr\left(A_\eta(\phi_{\sam{Y}{n}};\sam{Y}{n})+\frac{1}{\epsilon} V> 1+\frac{\log(2/\delta)}{\epsilon}\right)\right)+\delta\\
&=e^{2\epsilon}\Pr\left( \widetilde{T}(\sam{Y}{n})\in B\right)+\delta .
\end{align*}
The first inequality comes from the fact that we get $(2\epsilon,\delta)$ differential privacy if $B$ is restricted to $\mathscr{B}(\rdd)$ and the fact that $A_\eta(\phi_{\sam{Y}{n}};\sam{Y}{n})+\frac{1}{\epsilon} V$ is $\epsilon$-differentially private. 

Now, suppose that $V,\ a_\delta$ and $b_\delta$ correspond to the Gaussian version of PTR. Then, following the same steps as for the Laplace version gives, for $B\in \mathscr{B}(\rdd)$,
\begin{align*}
\Pr\left( \widetilde{T}(\sam{X}{n})\in B\right)\leq e^{2\epsilon}\Pr\left( \widetilde{T}(\sam{Y}{n})\in B\right)+\delta,
\end{align*}
when $\norm{\phi_{\sam{X}{n}}-\phi_{\sam{Y}{n}}}_{\infty}<\eta$. 
When $\norm{\phi_{\sam{X}{n}}-\phi_{\sam{Y}{n}}}_{\infty}\geq \eta$,
\begin{align*}
    \Pr\left( \widetilde{T}(\sam{X}{n})\in B\right)&\leq \Pr\left( Z\geq \sqrt{2\log(1.25/\delta)}\right) \leq \delta .
\end{align*}
We then have that 
$$\Pr\left(\widetilde{T}(\sam{X}{n})\in B\right)\leq e^{2\epsilon}\Pr\left(\widetilde{T}(\sam{Y}{n})\in B\right)+\delta.$$
Again, we need to include sets of the form $B=B'\cup\{\perp\}$, where $B'\in \mathscr{B}(\rdd)$. 
Consider 
\begin{align*}
\Pr\left( \widetilde{T}(\sam{X}{n})\in B\right)&=\Pr\left( \widetilde{T}({\sam{X}{n}})\in B'\right)+ \Pr\left(A_\eta(\phi_{\sam{X}{n}};\sam{X}{n})+\frac{\sqrt{2\log(1.25/\delta)}}{\epsilon} Z> 1+\frac{2\log(1.25/\delta)}{\epsilon}\right) \\
&\leq e^{2\epsilon}\left(\Pr\left( \widetilde{T}(\sam{Y}{n})\in B'\right)+ \Pr\left(A_\eta(\phi_{\sam{Y}{n}};\sam{Y}{n})+\frac{\sqrt{2\log(1.25/\delta)}}{\epsilon} Z> 1+\frac{2\log(1.25/\delta)}{\epsilon}\right) \right)+2\delta\\
&=e^{2\epsilon}\Pr\left( \widetilde{T}(\sam{Y}{n})\in B\right)+2\delta .
\end{align*}
\end{proof}

\begin{proof}[Proof of Theorem \ref{thm::PTR_EM}]
The only difference, from the previous proofs of results in this section is that we do not have a prior $\pi.$ 
The assumed existence of the measures $Q_{\sam{X}{n}}$ and $$Q(A)=\int_A\frac{\exp(-K\frac{\phi(v)}{2})}{\int_{\rdd}\exp(-K\frac{\phi(v)}{2})dv}dv$$
remedy this. 
The proofs of Lemma \ref{lem::exp_m} and Lemma \ref{lem::exp_m2} would then imply that $\widetilde{T}({\sam{X}{n}})$ satisfies the above convergence results if $P(\widetilde{T}({\sam{X}{n}})=\perp)\rightarrow 0$ as $n\rightarrow\infty$. 
This statement is, however, assumed by the theorem. 
\end{proof}
\begin{proof}[Proof of Theorem \ref{thm::c_PDM}]
First, note that because $\sup_{u}|\med(F_u)|<\infty$, $\sup_{u}\iqr(F_u)<\infty$ and $\inf_{u}\iqr(F_u)>0$, we have that 
\begin{equation*}
    O_2\left(x, F\right) \geq \frac{\norm{x}-\sup _{\|u\|=1} \mu\left(F_u\right)}{\sup _{\|u\|=1} \sigma\left(F_u\right)},
\end{equation*}
see page 1477 of \citep{zuo2003}. 
It immediately follows that
\begin{align*}
\int_{\mathbb{R}^d}\exp\left(-\frac{\epsilon_n  O_2(v;F)}{2\eta_n }\right)dv &\leq C_1\int_{\mathbb{R}^d}\exp\left(-\frac{\epsilon_n  \norm{v}}{2\eta_n }\right)dv<\infty.
\end{align*}
For the density 
\begin{align*}
\int_{\mathbb{R}^d}\exp\left(-\frac{\epsilon_n  \mathcal{O}_2(v,\sam{X}{n};M_n)}{2\eta_n }\right)dv &= \int_{\norm{v}<M_n}\exp\left(-\frac{\epsilon_n  \mathcal{O}_2(v,\sam{X}{n};M_n)}{2\eta_n }\right)dv<\infty.
\end{align*}
We then have both
$$\frac{\exp\left(-\frac{\epsilon_n  \mathcal{O}_2(v,\sam{X}{n};M_n)}{2\eta_n }\right)}{\int_{\mathbb{R}^d}\exp\left(-\frac{\epsilon_n  \mathcal{O}_2(v,\sam{X}{n};M_n)}{2\eta_n }\right)dv} \qquad\text{and}\qquad \frac{\exp\left(-\frac{\epsilon_n  O_2(v;F)}{2\eta_n }\right)}{\int_{\mathbb{R}^d}\exp\left(-\frac{\epsilon_n  O_2(v;F)}{2\eta_n }\right)dv}$$
are valid Lebesgue density functions. 
We need to show that $\mathcal{O}_2(v,\sam{X}{n};M_n)$ satisfies assumptions 2-4 of Lemma \ref{lem::exp_m}. 
Assumptions 2 and 3 hold on account of $\sup_{u}|\med(F_u)|<\infty$, $\sup_{u}\iqr(F_u)<\infty$ and $\inf_{u}\iqr(F_u)>0$, \citep[see Remark 2.5 of][]{zuo2003}. 
Assumption 4 is implied by the theorem assumptions. 
We need to show that 
$$\Pr\left(A_{\eta_n}(\phi_{\sam{X}{n}};\sam{X}{n})+V \frac{a_{\delta_n}}{\epsilon_n}\leq \frac{b_{\delta_n}}{\epsilon_n}+1\right)\rightarrow 0.$$ 
First, suppose that $\norm{x}<M_n$. 
From proof of Theorem \ref{thm::pd_con}, we have directly that
$$\limn\Pr\left(A_{\eta_n}(\phi_{\sam{X}{n}};\sam{X}{n})+V\frac{a_{\delta_n}}{\epsilon_n}\leq 1+\frac{\log(2/\delta_n)}{\epsilon_n}\right)=0.$$
The fact that $\norm{x}<M_n$ allows us to bounded the remainder in the Taylor series expansion used in the proof of Theorem \ref{thm::pd_con}. 
Now, suppose that $\norm{x}>M_n$, which implies that $O^u_\ell(x;\sam{X}{n})-O^u_\ell(x;\sam{Y}{n})=0$, which implies that 
$$\Pr\left(A_{\eta_n}(\phi_{\sam{X}{n}};\sam{X}{n})+V\frac{a_{\delta_n}}{\epsilon_n}\leq 1+\frac{\log(2/\delta_n)}{\epsilon_n}\right)=0.$$
The conditions of Theorem \ref{thm::PTR_EM} are satisfied and the result follows. 
\end{proof}

\end{document}